\documentclass[12pt,reqno]{amsart}
 \usepackage{amsmath, amsthm, amsfonts, amssymb, color}
\usepackage{amsfonts, amsmath}
\usepackage[colorlinks=true,
            linkcolor=blue,
            citecolor=blue,
            urlcolor=blue]{hyperref}
 \usepackage{enumitem}
\usepackage{mathrsfs} 
\newcommand{\ID}{{\mathcal G}}
\newcommand{\ir}{\lim_{a_2\to0\atop b_2\to\infty}\lim_{R_2\to\infty}\lim_{a_1\to0\atop b_1\to\infty}\lim_{R_1\to\infty}\int_{a_2}^{b_2}\int_{\IR_{2}}\int_{a_1}^{b_1}\int_{\IR_{1}}}
\newcommand{\IR}{\mathcal{G}}
\newcommand{\II}{\mathcal{G}_2\times\mathbb{R}_+}
\newcommand{\III}{\mathcal{G}_1\times\mathbb{R}_+}

\newcommand{\IT}{{\mathcal G}_{1}\times {\mathcal G}_{2}}
\newcommand{\n}{\widetilde{\nabla}}
\newcommand{\N}{\widetilde{\Delta}}
\newcommand{\B}{\mathscr{B}}

\newcommand{\X}{\mathcal X
 }
 \newcommand{\p}{\psi\left(\frac{h}{R}\right)}

\begin{document}

\newtheorem{theorem}{Theorem}[section]
\newtheorem{prop}[theorem]{Proposition}
\newtheorem{lemma}[theorem]{Lemma}
\newtheorem{definition}[theorem]{Definition}
\newtheorem{corollary}[theorem]{Corollary}
\newtheorem{example}[theorem]{Example}
\newtheorem{remark}[theorem]{Remark}
\newcommand{\ra}{\rightarrow}
\renewcommand{\theequation}
{\thesection.\arabic{equation}}
\newcommand{\ccc}{{\cal C}}

\allowdisplaybreaks

\def\HAL { H^1_{{at}, L_1, L_2, M}(\mathcal{G}_1\times\mathcal{G}_2)}
 \def\HSL { H^1_{\L_1, \L_2, S}(\IT) }
\def\HL  {   H^1_{\L_1, \L_2}(\IT) }
\def\HLD  {  H^{\rm 1}_{L} \cap L^2   }
\def\RR{\mathbb R}
\def\Rn{{\mathbb R}^n}
\def\Rm{{\mathbb R}^m}
\def\S{\mathcal{S}_{\L_1,\L_2}}
\def\SP{\mathcal{S}_{P,\L_1,\L_2}}
\def\G{{\mathcal G}_{1}}
\def\H{{\mathbb G}_{2}}
\def\l{\ell}
\def\g{{\bf g}}
\def\h{{\bf h}}
\def\ud{\,{\rm d}}
\def\t{{\rm{d\bf t}}}
\def\L{\mathcal{L}}
\def\R{\mathcal{R}}
\def\Sn{\mathcal{S}_{\n}}
\def\Snp{\mathcal{S}_{\n, \L_1,\L_2}}
\def\Ss{\mathcal{S}_{P,\L}}
\pdfstringdefDisableCommands{%
  \def\S{S}%
  \def\L{L}%
  \def\R{R}%
}

 \medskip

\title  [Fefferman--Stein type inequalities and applications]
{Fefferman--Stein type inequalities via area and maximal functions for Schr\"odinger operators with applications}

\author{Ji Li, \ Wanjun Li,\  Liang Song  and \   Lixin Yan}
\address{
Ji Li, School of Mathematical and Physical Sciences, Macquarie University, NSW 2109, Australia}
\email{
ji.li@mq.edu.au}
\address{
Wanjun Li, Department of Mathematics, Sun Yat-sen (Zhongshan) University, Guangzhou, 510275, P.R. China}
\email{
 liwj323@mail2.sysu.edu.cn}
\address{
Liang Song, Department of Mathematics, Sun Yat-sen (Zhongshan) University, Guangzhou, 510275, P.R. China}
\email{
songl@mail.sysu.edu.cn
}
\address{
Lixin Yan, Department of Mathematics, Sun Yat-sen (Zhongshan) University, Guangzhou, 510275, P.R. China}
\email{
mcsylx@mail.sysu.edu.cn
}



\date{}

\medskip

\begin{abstract}
In this paper, we establish a Fefferman--Stein inequality in terms of area function and non-tangential maximal function associated with the Schr\"odinger operator
$
\L = -\Delta + V$
on stratified Lie groups $\mathcal G$, where $\Delta$ denotes the sub-Laplacian on $\mathcal G$ and $V$ is a nonnegative locally integrable function. As an application, we extend this inequality to the tensor product $\mathcal G_1 \times \mathcal G_2$ of two stratified Lie groups and develop atomic decompositions associated with the Schr\"odinger operator for functions in the Orlicz space
$
L\log^{+}L(\mathcal G_1 \times \mathcal G_2).
$
Using these atomic decompositions, we further prove weak-type endpoint estimates for the area integral operator and the Riesz transforms associated with the Schr\"odinger operator on
$
L\log^{+}L(\mathcal G_1 \times \mathcal G_2),
$
thereby extending the celebrated result of R.\,Fefferman and E.M.\,Stein \cite{FSt1982} to the setting of singular integrals with non-smooth kernels. 
\end{abstract}

\maketitle


\section{Introduction and main results}
\setcounter{equation}{0}

Consider a harmonic function  $u(x, t)$  on the upper half-space $\mathbb R^{n}\times(0,\infty)$.
The non-tangential maximal function 
$$
u^*(x)= \sup_{|x-y|<t}|u(y,t)|
$$ 
and the area integral 
$$S(u)(x)=\left(\int_{|x-y|<t}|\nabla u(y,t)|^2 t^{1-n}\,\ud y\ud t\right)^{1/2}
$$
are two fundamental tools in the theory of singular integrals and the related function spaces. A celebrated result of C. Fefferman and E.M. Stein (\cite[Theorem 8]{FS1972}) states that, under the condition $u(x,t)\to0$ as $t\to\infty$, 
$$\|u^*\|_{L^p(\mathbb R^n)}\approx\|S(u)\|_{L^p(\mathbb R^n)},  \quad \quad 0<p\leq 1. 
$$
 (The case $1<p<\infty$  was proved earlier in \cite{Estein1958}).
The core of their proof relies on the following distributional inequality  (\cite[(7.2)]{FS1972}): for all $\lambda>0$,
\begin{align}\label{good lambda FS}
|\{x\in\mathbb R^n: S(u)(x)>\lambda\}|
 \lesssim|\{x\in\mathbb R^n: u^*(x)>\lambda\}|
+\cfrac{1}{\lambda^2}\int_0^\lambda r|\{x\in\mathbb R^n: u^*(x)>r\}|\ud r
\end{align}
and the corresponding inequality with the roles of $u^*$ and $S(u)$ interchanged (here  the implicit constant is independent of $\lambda$). This Fefferman--Stein inequality \eqref{good lambda FS} plays an important role in harmonic analysis and partial differential equations.

A significant recent breakthrough was achieved by Hofmann, Kenig, Mayboroda, and Pipher in \cite{HKMP2015}. They established a Fefferman--Stein type inequality \eqref{good lambda FS}, enabling norm comparisons between the area function and the non-tangential maximal function associated with divergence-form elliptic operators $\L$ having real, $t$-independent, and non-symmetric coefficients. As a consequence, they solved the $L^p$ Dirichlet problem for such operators \cite[Theorem 1.23]{HKMP2015}.

Inspired by the aforementioned results, we investigate analogous Fefferman--Stein type inequalities \eqref{good lambda FS} for the area function and the non-tangential maximal function associated with Schr\"odinger operators $\L = -\Delta + V$ on a stratified Lie group $\IR$ with homogeneous dimension $Q$ and a homogeneous metric $\rho$. Here $\Delta$ is the standard sub-Laplacian on $\IR$ and $V$ is a nonnegative locally integrable function. The relevant notation and settings are described in Sections \ref{Sec2.1} and \ref{Sec2.2}; for the general framework of stratified Lie groups, we refer to Folland and Stein \cite{FoSt}.

Let $f\in C_0^\infty(\IR)$, the space of infinitely differentiable functions with compact support on $\IR$. 
The area function and the non-tangential maximal function associated with the Schrödinger operator $\L$ are defined  as follows. For $g\in \IR$ and $\beta>0$, set
\[
\Ss (f)(g) = \left( \iint_{\Gamma(g)} \left|  t\sqrt{\L} e^{-t\sqrt{\L}}   \left(f\right)(h) \right|^2 \frac{\ud h \ud t}{t^{Q+1}} \right)^{1/2}
\]
and
\[
N^\beta_\L(f)(g) = \sup_{(h,t)\in \Gamma^\beta(g) } |e^{-t\sqrt{\L}}f(h)|, 
\]
where $\Gamma^\beta(g) := \{(h,t)\in\IR\times(0,\infty): \rho(g,h)<\beta t\}$ is the cone of aperture $\beta$ with vertex $g$.

The first main result of this paper is as follows.
\begin{theorem}\label{single fefferman-Stein}
 There exist $C>0$ and $\beta>1$ such that for all $f\in C_0^\infty(\IR)$ and  $\lambda>0$, there holds
\begin{equation}\label{main 1.2}
\begin{aligned}
&\left|\{g\in\IR:\,\Ss(f)(g)>\lambda\}\right|\\
&\qquad\leq C  \left|\{g\in\IR:\,N^\beta_\L(f)(g)>\lambda\}\right|+
\frac{C}{\lambda^2}\int_{\{g\in\IR:\,N^\beta_\L(f)(g)\leq \lambda\}}\left|N^\beta_\L(f)(g)\right|^2\ud g.
\end{aligned}
\end{equation}

\end{theorem}

We emphasize that proving inequality \eqref{main 1.2} is highly nontrivial; it does not follow directly from any known results,  such as those in
\cite{FS1972,merryfield,HKMP2015,CowlingFanLiyan2025, CLLP2025,songJAM2011}.  The main difficulty  comes from the nonnegative potential \(V\) and the structure of the stratified Lie group.  On one hand, the geometric 
and algebraic structure of the stratified Lie group precludes the direct use of classical differentiation arguments that are standard in the Euclidean setting.  On the other hand, the presence of the nonnegative potential $V$ complicates the good-$\lambda$ inequalities.  
Even in the
Euclidean case, \eqref{main 1.2} was not stated explicitly before, though we point out that it can be derived from intermediate estimates in the proof of \cite{songJAM2011}.
However, the argument in \cite{songJAM2011} relies on the result and proof of Merryfield \cite{merryfield}. One key step in Merryfield's
 approach for the standard Laplacian ($V\equiv 0$) uses a nice bump function $\phi$ on $\mathbb R$--specifically, a compactly supported, smooth, positive even function--which satisfies the nice property $\frac{\partial}{\partial t}(\phi_t\ast f)=-\frac{\partial}{\partial x}(\rho_t\ast f)$, where $\rho(x)=x\phi(x)$.   Such constructions are not available on general stratified Lie groups. We refer to the last section for full details.

To overcome these difficulties, we develop a different approach inspired by \cite{liji2023}. 
Instead of relying on the standard bump functions used in \cite{merryfield,songJAM2011}, we construct a smooth  function  that  captures the range of $e^{-t\sqrt{-\Delta}}(\chi_{E_\lambda})$, that is, the Poisson semigroup associate with the sub-Laplacian $\Delta$ but not with $\L$ (see Section 3.2). This alternative choice is intrinsically compatible with the structure of stratified Lie groups and successfully accommodates the global behavior induced by the potential. We then integrate this with a delicate truncation and approximation procedure using compactly supported smooth cutoffs. This combined approach provides a rigorous framework for obtaining the required local estimates, a step that is indispensable for our subsequent analysis on product spaces.

We further derive a Fefferman--Stein inequality on product spaces, extending the  result of Merryfield \cite{merryfield} (see also recent progress in \cite{liji2023,CLLP2025}). We define the product area function and the non-tangential maximal function related to the Schr\"odinger operator $\L_i=-\Delta_i+V_i,$ where $V_i\in L^1_{loc}(\IR_i)$ for $i=1,2$, as follows: for $\g=(g_1,g_2)\in \IT$,

\begin{align}\label{esf}
\SP f(\g) = \left( \iint_{\Gamma(\g)} \left| \left( t_1\sqrt{\L_1} e^{-t_1\sqrt{\L_1}} \otimes t_2\sqrt{\L_2} e^{-t_2\sqrt{\L_2}} \right) f({\h}) \right|^2 \frac{{\ud h_1\ud h_2 \ud t_1\ud t_2}}{t_1^{Q_1+1} t_2^{
Q_2+1}} \right)^{1/2}
\end{align}
and
$$
N^\beta_{\L_1,\L_2}(f)(\g)=\sup_{{\bf h}\in\Gamma^\beta(\g)} \left|\left(   e^{-t_1\sqrt{\L_1}} \otimes  e^{-t_2\sqrt{\L_2}} \right) f({\bf h}) \right|, 
$$
where $Q_i$ is the homogeneous dimension of $\IR_i$ and $\rho_i$ is the homogeneous norm of $\IR_i$, ${\h} = (h_1, h_2)$, ${\bf t} = (t_1, t_2)$, and the product cone $\Gamma^\beta(\g)$ is defined by $\Gamma^\beta(\g) := \Gamma^\beta(g_1)\times \Gamma^\beta(g_2)$, with each marginal cone given by
\begin{equation*}
    \Gamma^\beta(g_i) := \{(h_i,t_i)\in \IR_{i}\times(0,\infty) : \rho_i(g_i,h_i)<\beta t_i\}, \quad \text{for } i=1,2.
\end{equation*}

We then extend the method of Theorem \ref{single fefferman-Stein} to the two-parameter product space.
\begin{theorem}\label{prod F-S thm}
There exist $C>0$ and $\beta>1$ such that for all $f\in C_0^\infty(\IT)$ and $\lambda>0$,
\begin{equation}\label{FS-Product space}
\begin{aligned}
&\left|\{\g\in\IT:\,\SP(f)(\g)>\lambda\}\right|\\
&\lesssim  \left|\{\g\in\IT:\,N^\beta_{\L_1,\L_2}(f)(\g)>\lambda\}\right|+\frac{C}{\lambda^2}\int_{\{\g:\, N^\beta_{\L_1,\L_2}(f)(\g)\leq \lambda\}}\left|N^\beta_{\L_1,\L_2}(f)(\g)\right|^2\,\rm{d}\g.
\end{aligned}
\end{equation}

\end{theorem}


Crucially, the Fefferman--Stein inequality \eqref{FS-Product space} yields the weak-type endpoint estimate for the product area function $\SP(f)$ (see Corollary \ref{lemma auxiliary}):
\begin{align*}
\left|\left\{ \g \in \IT : |\SP(f)(\g)| > \lambda \right\}\right|
\leq C \left\|\frac{f}{\lambda}\right\|_{L \log^+ L(\IT)},
\end{align*}
where $$
\|f\|_{L\log^+ L(\IT)}:=\int_{\IT}|f(\g)|\log(e+|f(\g)|)\,{\rm d\g}.
$$
This endpoint bound, in turn, provides the fundamental ingredient needed to establish an atomic decomposition associated with the Schr\"odinger operators $\L_1$ and $\L_2$ for the Orlicz space $L\log^+ L(\IT)$. It is worth pointing out that, while atomic decompositions in both the classical and operator-adapted settings are usually established on Hardy spaces or their operator-associated counterparts (see, e.g., \cite{CF1980,CF1982,HLMMY}), the atomic decomposition obtained in this paper is developed specifically for the Orlicz space $L\log^{+}L(\IT)$. 


\begin{theorem}\label{lemma L log L atom}
Let
 $f\in L\log^+L(\IT)\cap L^2(\IT)$. Then we can
decompose $f$ as follows.

\begin{eqnarray}
f=\sum_ka_k,
\end{eqnarray}

\noindent where the $a_k$ are atoms with the following properties: 
each $a_k$ is associated to an open set $\mathcal{O}_k \in\IT$ with finite measure with
$|\mathcal{O}_k|\leq C\|2^{-k}f\|_{L\log^+L(\IT)}$, and

\medskip $(1)$ $a_k$ is supported in the enlargement of $\mathcal{O}_k$, that is 
supp $a_k \subset \bigcup\limits_{R\subset \mathcal{O}_k}3R$;

\medskip $(2)$ each $a_k$ can be further decomposed as $
a_k=\sum_{S\in \mathscr{M}(\mathcal{O}_k)}a_{k,S} $ for any positive integer $M$, and there exists a function
$b_{k,S}\in \mathcal {D}(\L_1^M\otimes \L_2^M)$ such that

\medskip

\hskip.5cm $(i)$  $a_{k,S}=(\L_1^M\otimes \L_2^M)b_{k,S}$;

\medskip

\hskip.5cm $(ii)$  {\rm supp}$(\L_1^{i_1}\otimes
\L_2^{i_2})b_{k,S}\subset 3S$, $0\leq i_1,i_2\leq M$;

\medskip

\hskip.5cm $(iii)$ $\|a_k\|_{L^2(\IT)}^2\leq C2^{2k}\|2^{-k}f|\|_{L\log^+L(\IT)}$ and for $0\leq
i_1,i_2\leq M$

$$ \sum_{S\in \mathscr{M}(\mathcal{O}_k)} \l(I)^{-4M}\l(J)^{-4M}
\|(\l(I)^2\L_1)^{i_1}\otimes (\l(J)^2\L_2)^{i_2}
b_{k,S}\|_{L^2(\IT )}^2\leq C2^{2k}\|2^{-k}f\|_{L\log^+L(\IT)}, $$
where $\L_1$ and $\L_2$ are Schr\"odinger operators as above, and $\ell(I)$, $\ell(J)$ denote the side-lengths of  cubes $I$ and $J$, respectively. The detailed geometric definition of the maximal dyadic rectangles $S=I\times J\in\mathscr{M}(\mathcal{O}_k)$ is postponed to Sections \ref{sec 2.4} and \ref{sec 2.5}.
\end{theorem}


The operator-adapted atomic decomposition established in Theorem \ref{lemma L log L atom} is specifically designed to bypass the limitations of classical harmonic analysis in this setting. Due to the presence of the potentials $V_1$ and $V_2$, the area function and double Riesz transform associated with $\L_1$ and $\L_2$ possess non-smooth kernels. This places them firmly outside the classical product Calder\'on--Zygmund framework formulated in \cite{FSt1982}. Consequently, testing these operators against the standard $L\log^+L$ atoms developed in prior works (e.g., \cite{FSt1982, fefferman1986, CowlingFanLiyan2025, CLLP2025}) fails to yield the required bounds. 

By instead employing our newly established operator-adapted atoms, we can successfully capture the necessary cancellations. As a direct application, we obtain the following weak-type endpoint estimates.

\begin{theorem}\label{thm1}
 Consider an operator $T$ which is either:
(1) the Littlewood--Paley area function $\S$,  or
(2) 
the double Riesz transform $\nabla_{g_1} \L_1^{-1/2}\otimes \nabla_{g_2} \L_2^{-1/2}$.
Then $T$ satisfies the following weak-type endpoint estimate on the Orlicz space $L(\log^+ L)(\IT)$: for every $\lambda>0$,
\begin{align}\label{estimate1.2}
\left|\left\{ \g \in \IT : |Tf(\g)| > \lambda \right\}\right|
\leq C \left\|\frac{f}{\lambda}\right\|_{L \log^+ L(\IT)},
\end{align}
where
\begin{align}\label{double S-heat}
\S(f)(\g)=\bigg(\iint_{\Gamma(\g) }\big|\big(t_1^2\L_1 e^{-t_1^2\L_1}\otimes
t_2^2\L_2e^{-t_2^2\L_2}\big)f(\h)\big|^2
\frac{\ud\h\,\t}{t_1^{Q_1+1} t_2^{Q_2+1}}\bigg)^{1/2}.
\end{align}
\end{theorem}
\medskip

Our results can be regarded as a substantial extension of \cite{CLLP2025}. Indeed, the results in \cite{CLLP2025} are confined to the classical case $V \equiv 0$, whereas the present paper extends the theory to Schr\"odinger operators with nontrivial potentials. Due to the presence of the potential $V$, the operators considered in this paper possess nonsmooth kernels and therefore fall outside the Calder\'on--Zygmund framework, distinguishing our results from those in \cite{fefferman1986,Feff1986,F1987,HLCL2010,CLLP2025}, where only operators with smooth kernels are studied. In addition, we establish global endpoint estimates, which differ fundamentally from the local ones obtained in \cite{fefferman1986}.
Furthermore, for Riesz transforms associated with Schr\"odinger operators, existing works, e.g., \cite{HLMMY,SY2010}, have  proved  boundedness from operator-adapted Hardy spaces  to the classical Hardy space. In contrast, we obtain endpoint estimates directly on the Orlicz space  $L\log^{+}L$ in the product setting.

The layout of the paper is as follows.  In Section \ref{section-backgound}, we review the necessary background on heat kernels and finite propagation speed for the wave equation.
 In Section \ref{Single proof}, we prove Theorem \ref{single fefferman-Stein}, which is the single-parameter Fefferman--Stein inequality associated with  Schr\"odinger operators. The product version of the Fefferman--Stein inequality associated with Schr\"odinger operators, Theorem \ref{prod F-S thm}, is proved in
 Section \ref{prouct proof}.
 In Section \ref{atomic decomposition}, we apply  Theorem \ref{prod F-S thm} to prove Theorem \ref{lemma L log L atom}, which provides  an atomic decomposition of the global  $L\log^+L(\IT)$  space associated with the Schr\"odinger operators $\L_1$ and $\L_2$.
 As an application of Theorem \ref{lemma L log L atom},    Theorem \ref{thm1} is proved in Section \ref{application}. In the last section we provide some remarks on the corresponding Fefferman--Stein type inequality on Eucldiean spaces.

Throughout, the letters ``$c$" and ``$C$" will denote (possibly
different) constants that are independent of the essential
variables. $A \lesssim B$ denotes that there exists a constant $C>0$ such that $A \leq CB$;  
$A \approx B$ means $A \lesssim B$ and $B \lesssim A$. We write $\chi_E$ for the indicator function of a set $E$. Additionally, we write $\mathbb{R}_+$ for the open interval $(0,\infty)$.

\vskip 1cm

\section{Preliminaries}\label{section-backgound}
\setcounter{equation}{0}

\subsection{Stratified nilpotent Lie groups}\label{Sec2.1}
Recall that a connected, simply connected nilpotent Lie group $\mathcal{G}$ is said to be stratified if its left-invariant Lie algebra $\mathfrak{g}$ (assumed real and of finite dimension) admits a direct sum decomposition
\begin{align*}
\mathfrak{g} = \bigoplus_{i = 1}^k V_i \  \mbox{where $[V_1, V_i] = V_{i + 1}$ for $i \leq k - 1$.}
\end{align*}
One identifies $\mathfrak{g}$ and $\mathcal{G}$ via the exponential map
\begin{align*}
\exp: \mathfrak{g} \longrightarrow \mathcal{G},
\end{align*}
which is a diffeomorphism.

We fix a (bi-invariant) Haar measure $dg$ on $\mathcal G$  once and for all (which is just the lift of Lebesgue measure on $\frak g$ via $\exp$).

There is a natural family of dilations on $\frak g$ defined for $r > 0$ as follows:
\begin{align*}
\delta_r \left( \sum_{i = 1}^k v_i \right) = \sum_{i = 1}^k r^i v_i, \quad \mbox{with $v_i \in V_i$}.
\end{align*}
This allows the definition of dilation on $\mathcal{G}$, which we still denote by $\delta_r$.

We choose once and for all a basis $\{\X_1, \cdots, \X_n\}$ for $V_1$ and consider the sub-Laplacian $\Delta = \sum_{j=1 }^n \X_j^2 $. Observe that $\X_j$ ($1 \leq j \leq n$) is homogeneous of degree $1$ and $\Delta$ of degree $2$ with respect to the dilations in the sense that:
\begin{align*}
&\X_j \left( f \circ \delta_r \right) = r \, \left( \X_j f \right) \circ \delta_r, \qquad  1 \leq j \leq  n, \  r > 0, \ f \in C^1; \\
&\delta_{\frac{1}{r}} \circ \Delta \circ \delta_r = r^2 \, \Delta, \hskip 2cm  r > 0.
\end{align*}

Let $Q$ denote the homogeneous dimension of $\mathcal{G}$, namely,
\begin{align}\label{homo dimension}
Q= \sum_{i=1}^k i\, {\rm dim} V_i.
\end{align}
Next we recall the homogeneous norm $\rho$ on $\mathcal G$ (See, for example, \cite{FoSt}), which is  
defined as a continuous function
$g\to \rho(g)$ from $\mathcal G$ to $[0,\infty)$, which is $C^\infty$ on $\mathcal G\setminus \{o\}$
and satisfies the following properties:
\begin{enumerate}
\item[(a)] $\rho(g^{-1}) =\rho(g)$;
\item[(b)] $\rho({ \delta_r(g)}) =r\rho(g)$ for all $g\in \mathcal G$ and $r>0$;
\item[(c)] $\rho(g) =0$ if and only if $g=o$.
\end{enumerate}
For the existence (also the construction) of the homogeneous norm $\rho$ on  $\mathcal G$,
we refer to  \cite[Chapter 1, Section A]{FoSt}. For convenience,
we set
\begin{align*}
\rho(g, g') = \rho(g'^{-1} \circ g) = \rho(g^{-1} \circ g'), \quad \forall g, g' \in \mathcal{G}.
\end{align*}
This defines a quasi-distance in the sense of  Coifman-Weiss, namely, there exists a constant $C \geq1 $ such that
 \begin{align} \label{qdr}
 \rho(g_1, g_2) \leq C \, \left( \rho(g_1, g') + \rho(g', g_2)   \right), \qquad \forall g_1,\, g_2,\, g' \in  \mathcal{G}.
\end{align}
In the sequel, we fix a homogeneous norm $\rho$ on  $\mathcal G$.


In the sequel, for $g\in\mathcal G$ and $r>0$, {$B(g,r)$
denotes the open ball defined by $\rho$.
Without} loss of generality, we may suppose that the Haar measure has been normalized such that the measure of $B(o, 1)$, $|B(o, 1)| = 1$.

Now we introduce some estimates for the kernel. Let $H_t$ ($t> 0$)  be the heat kernel (that is, the integral kernel of $e^{t \Delta}$)
on $\mathcal G$. For convenience, we set $H_t(g) = H_t(g, o)$ (that is, in this article, for a convolution operator, we will identify the integral kernel with the convolution kernel) and $H(g) = H_1(g)$.

Recall that  for $t>0$ and $ g \in \mathcal G$ (see, e.g., \cite{FoSt}),
\begin{align} \label{hkp1}
H_t(g) = t^{-\frac{Q}{2}} H(\delta_{\frac{1}{\sqrt{t}}}(g)),
\end{align}
and the kernel $H_t$ of $e^{t\Delta}$ satisfies  (see \cite{JS1986})
\begin{align}\label{Delta estmates}
    0<H_t(g)\leq Ct^{-\frac{Q}{2}}e^{-c\frac{\rho(g)^2}{t}}.
\end{align}

Let $p_t(g,h)$ denote the kernel of the Poisson semigroup $e^{-t\sqrt{-\Delta}}$. By the subordination formula, we obtain
\begin{align*}
    0< p_t(g,h)\lesssim \cfrac{t}{(t^2+\rho(g,h)^2)^{(Q+1)/2}}.
\end{align*}

The kernel of the $j^{\mathrm{th}}$  Riesz transform $\X_j (-\Delta)^{-\frac{1}{2}}$ ($1 \leq j \leq  n$) is written simply as $K_j(g, g') = K_j(g'^{-1} \circ g)$. It is well-known that
\begin{align} \label{kjs}
K_j \in C^{\infty}(\mathcal G \setminus \{o\}), \ K_j(\delta_r(g)) = r^{-Q} K_j(g), \quad \forall \, g \neq o, \ r > 0, \ 1 \leq j \leq  n,
\end{align}
which can also be explained by \eqref{hkp1} and the fact that
\begin{align*}
K_j(g) = \frac{1}{\sqrt{\pi}} \int_0^{+\infty} t^{-\frac{1}{2}} \X_j H_t(g) \, \ud t  = \frac{1}{\sqrt{\pi}} \int_0^{+\infty} t^{- \frac{Q}{2} - 1} \left( \X_j H \right)(\delta_{\frac{1}{\sqrt{t}}}(g)) \, \ud t.
\end{align*}

\medskip

\subsection{Schr\"odinger operator}\label{Sec2.2}
Let $V \geq 0$ be a non-zero potential function in the space $L^1_{loc}(\mathcal{G})$ of locally integrable functions, where $n \geq 1$. We define the sesquilinear form $\B$ by
\[ \B(u,v) =\int_{\mathcal{G}} \nabla u \nabla v\ud g+ \int_{\mathcal{G}} Vuv\ud g   \]
with its domain given by
\[ \mathcal{D}(\B) = \bigg\{ u \in W^{1,2}(\mathcal{G}) : \int_{\mathcal{G}} V(g)|u(g)|^2\ud g < \infty \bigg\}. \]
Here $\nabla := (\X_1, \dots, \X_n)$ denotes the sub-gradient, and $W^{1,2}(\mathcal{G})$ is the corresponding Sobolev space consisting of functions $f \in L^2(\mathcal{G})$ such that $\nabla f \in L^2(\mathcal{G})$, where  $\nabla f$ is understood in the weak sense.
 The form $\B(\cdot,\cdot)$ is symmetric and closed. According to Simon's theorem, $\B(\cdot,\cdot)$ equals the form closure of its restriction to $C_0^\infty(\mathcal{G})$.

Denote by $\L$ the self-adjoint operator associated with  $\B(\cdot,\cdot)$. Its domain is defined as
\[ \mathcal{D}(\L) = \bigg\{ u \in \mathcal{D}(\B) : \text{there exists } v \in L^2(\mathcal{G}) \text{ such that } \B(u,\phi) =\int_{\mathcal{G}}v\phi \ud g \text{ for all } \phi \in \mathcal{D}(\B) \bigg\}. \]
Formally, $\L$ represents the Schr\"odinger operator $-\Delta + V$.  Let $H_t^{\L}$ denote the kernel of $e^{-t\L}$.  
 The Trotter product formula (see \cite{Goldstein1985}) yields,  for every $t>0$ and $g,\,h \in \mathcal{G}$,  
\begin{equation}\label{L Heat kernel}
     0\leq H_t^\L(g,h) \leq C t^{-Q/2} \exp\bigg(-c\frac{\rho(g,h)^2}{t}\bigg). 
     \end{equation}
Now denote  by $H_{t,k}^\L$ the kernels of the operator $(t^2\L)^ke^{-t^2\L}$ for $k=1,2,\cdots$.   Combining \cite [Theorem 6.17]{Ouhabaz2005}  with  \eqref{L Heat kernel}, we obtain
    \begin{equation}\label{estimate heat kernel}
        \left|H_{t,k}^\L(g,h)\right|\leq \cfrac{C_k}{t^Q} \exp\left(-c\cfrac{\rho(g,h)^2}{t^2}\right),
    \end{equation}
    for all $k\in {\mathbb N}$, $t>0$ and almost every $g,h\in\IR$.
    
Let $p_t^\L(g,h)$ denote the kernel of the Poisson semigroup $e^{-t\sqrt{\L}}$. 
Based on the subordination formula
\[
e^{-t\sqrt{\L}}=\cfrac{1}{2\sqrt{\pi}}\int_0^\infty \cfrac{te^{-t^2/4v}}{\sqrt{v}}e^{-v\L}\,\cfrac{{\rm d}v}{v},
\]
together with the doubling property of the measure and the estimate of $H^{\L}_t$ as above, we obtain \begin{align}\label{poisson bound estimate}
    |p_t^\L(g,h)|\lesssim \cfrac{t}{(t^2+\rho(g,h)^2)^{(Q+1)/2}},
\end{align}
for each $g,h\in\IR$ and $t>0$. 

Recall  that the Hardy--Littlewood maximal operator is defined as follows.
$$
\mathcal{M}(f)(g):=\sup\limits_{B\ni g} \left\{\cfrac{1}{|B|}\int_{B} |f(h)|\ud h: \  {\rm B \ is \  a \  ball \  of\ } {\mathcal G}  \right\}
$$ 
It is easy to check that for every $f\in L^2(\IR)$, 
\begin{align}\label{poisson controlled by H--L}
  |p_t(f)(g)|+|p_t^\L(f)(g)|\leq C_0 \inf_{z:\,\rho(z,g)<t}\mathcal{M}(f)(z), 
\end{align}
where $C_0$ is independent of the underlying variables.



\medskip

\subsection{Finite  propagation speed for the wave equation and spectral multipliers}\ Note that $\L$ is a non-negative, self-adjoint operator on $L^2(\mathcal{G})$. Let $E_{\L}(\lambda)$ denote its spectral decomposition. Then, for every bounded
Borel function $F:[0,\infty)\to{\mathbb{C}}$, one defines the bounded operator
$F(\L): L^2(\mathcal{G})\to L^2(\mathcal{G})$ by the formula
\begin{eqnarray}\label{e3.11}
F(\L):=\int_0^{\infty}F(\lambda)\,{\rm dE}_{\L}(\lambda).
\end{eqnarray}
In particular, the operator $\cos(t\sqrt{\L})$ is  well-defined and bounded
on $L^2(\mathcal{G})$. Moreover, it follows from \eqref{L Heat kernel} and   \cite[Theorem 3]{CS}
  that    there exists a finite,
positive constant $c_0$ with the property that the Schwartz
kernel $K_{\cos(t\sqrt{\L})}$ of $\cos(t\sqrt{\L})$ satisfies
\begin{eqnarray}\label{e3.12} \hspace{1cm}
{\rm supp} K_{\cos(t\sqrt{\L})}\subseteq
\big\{(g,h)\in {\IR\times\IR: \rho(g,h)\leq c_0 t\big\}.}
\end{eqnarray}
\noindent See also \cite{Sikora2004}. By the Fourier inversion
formula, whenever $F$ is an even, bounded, Borel function with its Fourier transform
$\hat{F}\in L^1(\mathbb{R})$, we have
\begin{eqnarray}\label{XCP}
F(\sqrt{\L})=(2\pi)^{-1}\int_{-\infty}^{\infty}{\hat F}(t)\cos(t\sqrt{\L})\ud t,
\end{eqnarray}
which, when combined with (\ref{e3.12}), gives
\begin{eqnarray}\label{e3.13} \hspace{1cm}
K_{F(\sqrt{\L})}(g,h)=(2\pi)^{-1}\int_{|t|\geq c_0^{-1}\rho(g,h)}{\hat F}(t)
K_{\cos(t\sqrt{\L})}(g,h)\ud t,\qquad \forall\,g,\,h\in\mathcal{G}.
\end{eqnarray}

The following lemma is very useful in the sequel.

\begin{lemma}[\cite{H}]\label{lemma finite speed} Let $\varphi\in C^{\infty}_0(\mathbb R)$ be
even, $\mbox{supp}\,\varphi \subset (-c_0^{-1}, c_0^{-1})$, where $c_0$ is
the constant in (\ref{e3.12}). Let $\Phi$ denote the Fourier transform of
$\varphi$. Then for every $k=0,1,2,\dots$, and for every $t>0$,
the kernel $K_{(t^2\L)^{k}\Phi(t\sqrt{\L})}(\cdot,\cdot)$ of the operator
$(t^2\L)^{k}\Phi(t\sqrt{\L})$ which was defined by the spectral theory, satisfies

\begin{eqnarray}
{\rm supp}\ \, K_{(t^2\L)^{k}\Phi(t\sqrt{\L})}(g,h) \subseteq
\Big\{(g,h)\in \IT: \rho(g,h)\leq t\Big\}.
\end{eqnarray}
\end{lemma}

\medskip

Next, for $s>0$, we define
$$
{\mathbb F}(s):=\Big\{\psi:{\Bbb C}\to{\Bbb C}\ {\rm measurable}: \ \
|\psi(z)|\leq C {|z|^s\over ({1+|z|^{2s}})}\Big\}.
$$
Then for any non-zero function $\psi\in {\Bbb F}(s)$, we have that
$\{\int_0^{\infty}|{\psi}(t)|^2\frac{\ud t}{t}\}^{1/2}<\infty$.
We denote  $\psi_t(z)=\psi(tz)$. It follows from the spectral theory
in \cite{Yo} that for any $f\in L^2({\mathcal G})$,

\begin{eqnarray}
\Big\{\int_0^{\infty}\|\psi(t\sqrt{\L})f\|_{L^2(\mathcal{G})}^2{\ud t\over t}\Big\}^{1/2}
&=&\Big\{\int_0^{\infty}\big\langle\,\overline{ \psi}(t\sqrt{\L})\,
\psi(t\sqrt{\L})f, f\big\rangle {\ud t\over t}\Big\}^{1/2}\nonumber\\
&=&\Big\{\big\langle \int_0^{\infty}|\psi|^2(t\sqrt{\L}) {\ud t\over t}f,
f\big\rangle\Big\}^{1/2}\nonumber\\
&\leq& \kappa \|f\|_{L^2({\mathcal G})}, \label{e2.155}
\end{eqnarray}

\noindent where $\kappa=\big\{\int_0^{\infty}|{\psi}(t)|^2
{\ud t/t}\big\}^{1/2},$ an estimate which will be often used in the sequel.
\medskip

\subsection{Systems of pseudodyadic cubes and products of stratified groups}\label{sec 2.4}

In the classical Euclidean setting, dyadic cubes are a fundamental tool in harmonic analysis. However, in stratified Lie groups or more general metric spaces, the traditional dyadic cubes are no longer applicable. To develop a theory of singular integrals in non-homogeneous spaces or product spaces, it becomes necessary to construct analogous systems of ``pseudodyadic cubes". The work of Hytönen and Kairema \cite{hk2012} provides such a construction in geometrically doubling metric spaces. In this section, we present a streamlined version of their construction and apply it to the product spaces of stratified groups.


\begin{theorem}[{\cite[Theorem 2.2]{hk2012}}]
\label{thm:hytonen-kairema}
Let $c_*$, $C_*$ and $\kappa$ be constants such that $0 < c_* \leq C_* < \infty$ and $12C_*\kappa \leq c_*$, and let $(\mathcal{G},\rho)$ be a metric stratified group. Then, for all $k \in \mathbb{Z}$, there exist families $\mathcal{Q}_k(\mathcal{G})$ of pseudodyadic cubes $Q$ with centres $z(Q)$, such that:

\begin{enumerate}
    \item $\mathcal{G}$ is the disjoint union of all $Q \in \mathcal{Q}_k(\mathcal{G})$, for each $k \in \mathbb{Z}$.
    \item $B(z(Q), c_*\kappa^{k}/3) \subseteq Q \subseteq B(z(Q), 2C_*\kappa^{k})$ for all $Q \in \mathcal{Q}_k(\mathcal{G})$.
    \item if $Q \in \mathcal{Q}_k(\mathcal{G})$ and $Q' \in \mathcal{Q}_{k'}(\mathcal{G})$ where $k \leq k'$, then either $Q \cap Q' = \emptyset$ or $Q \subseteq Q'$; in the second case, $B(z(Q), 2C_*\kappa^{k}) \subseteq B(z(Q'), 2C_*\kappa^{k'})$.
\end{enumerate}
\end{theorem}

We write $\mathcal{Q}(\mathcal{G})$ for the union of all $\mathcal{Q}_k(\mathcal{G})$, and call this a system of pseudodyadic cubes. Given a cube $Q \in \mathcal{Q}_k(\mathcal{G})$, we denote the quantity $\kappa^{k}$ by $\ell(Q)$, by analogy with the side-length of a Euclidean cube.

Now we introduce some notation for the product space that will be used
throughout the paper. 
The ambient space is the product of two stratified groups 
$\mathcal{G}_1$ and $\mathcal{G}_2$ with product structure. We follow Section~\ref{Sec2.1} and use the 
subscript \(i\) to distinguish the factors \(\mathcal{G}_i\), \(i = 1, 2\).
Points in $\mathcal{G}_1\times\mathcal{G}_2$ are written in boldface:
$\g = (g_1,g_2)$ and $\h = (h_1,h_2)$.

The first family of geometric objects we need are \emph{product balls}: 
sets of the form $B_1\times B_2$, where each $B_i$ is a ball in 
$\mathcal{G}_i$. We write $\mathscr{P}(\IT)$ for the collection of all 
such products. 
Using $\mathscr{P}(\IT)$, we define the \emph{strong maximal operator}
$\mathcal{M}_s$ is defined by
\begin{align*}
\mathcal{M}_s(f)(\g):=\sup\limits_{B_1\times  B_2\ni \g} \left\{\cfrac{1}{|B_1\times B_2|}\int_{B_1\times B_2} |f(\h)|\,{\rm d\h}:\,B_1\times B_2\in \mathscr{P}(\IT)\right\},
\end{align*}
for all $f\in L^1_{loc}(\IT)$.
It is clear that $\mathcal{M}_s$ is dominated by the composition of the 
Hardy--Littlewood maximal operators in the factors:
\[
\mathcal{M}_sf \leq \mathcal{M}_1\mathcal{M}_2(f) 
\quad\text{and}\quad 
\mathcal{M}_sf \leq \mathcal{M}_2\mathcal{M}_1(f),
\]
where $\mathcal{M}_i$, $i=1,2$, denotes the Hardy--Littlewood maximal 
operator acting on $\mathcal{G}_i$.
Since $\mathcal{M}_i,\,i=1,2$ are $L^p$ bounded for $1<p\leq \infty$, the iterated maximal operators and hence the strong  maximal operator
$\mathcal{M}_s$ are  $L^p$ bounded  for $1<p\leq \infty$.

Similar to \eqref{poisson controlled by H--L}, we also have that for every $f\in L^2(\IT)$, 
\begin{align}\label{product poisson controlled}
   |p_{t_1}p_{t_2}(f)(\g)|+ |p_{t_1}^{\L_1}p_{t_2}^{\L_2}(f)(\g)|\leq C_1 \inf_{\h:\,\rho_i(h_i,g_i)<t_i,i=1,2} \mathcal{M}_s(f)(\h),
\end{align}
where $C_1$ does not depend on the variables or on $f$.

The second family of geometric objects are \emph{dyadic rectangles}:
products $I\times J$, where $I\subset\mathcal{G}_1$ and 
$J\subset\mathcal{G}_2$ are pseudodyadic cubes from the Hyt\"onen--%
Kairema system. The family of all dyadic rectangles is denoted by 
$\mathscr{R}(\IT)$.

The following property links dyadic rectangles with the strong maximal 
operator. This lemma will be needed for the atomic decomposition later on.

\begin{lemma}[ {\cite[Lemma 3.1]{CLLP2025}}]\label{R cap Omega>1/2R}
    There exists a geometric constant $C_2>0$ such that
    \[
    R\subseteq \big\{\g\in\IT: 
    \mathcal{M}_s(\chi_U)(\g) > C_2\alpha \big\}
    \]
    for all measurable sets $U\subset\IT$, all $\alpha\in(0,1)$, and all 
    rectangles $R\in\mathscr{R}(\IT)$ satisfying
    \[
    |R\cap U| \geq \alpha |R|.
    \]
\end{lemma}


\subsection{Journe's covering lemma}\label{sec 2.5}

Let $U$ be an open subset of $\IT$ of finite measure. Denote by $\mathscr{M}(U)$   the set of all maximal subrectangles of $U$, and by $\mathscr{M}_i(U)$  the family of rectangles $R$ in $U$ which are maximal (in terms of inclusion) in the $i$th ``direction".  Given $R=Q_1\times Q_2\in\mathscr{M}_1(U)$, let $\tilde{Q}_2$ be the largest pseudodyadic cube containing $Q_2$ such that
\[
|\big(Q_1\times\tilde{Q}_2\big)\cap U|>\frac{1}{2}|Q_1\times\tilde{Q}_2|.
\]
Similarly, given $R=Q_1\times Q_2\in\mathscr{M}_2(U)$, let $\tilde{Q}_1$ be the largest pseudodyadic cube containing $Q_1$ such that
\[
|\big(\tilde{Q}_1\times Q_2\big)\cap U|>\frac{1}{2}|\tilde{Q}_1\times Q_2|.
\]
We then set
\[
\gamma_1(R):=\frac{\ell(\tilde{Q}_1)}{\ell(Q_1)}\quad\text{and}\quad\gamma_2(R):=\frac{\ell(\tilde{Q}_2)}{\ell(Q_2)}.
\]

The following lemma is a Journé-type covering lemma on the space $\IT$.

\begin{lemma}[{\cite[Lemma 2.2]{HLL2016}}]\label{Journetype}
Suppose that $U$ is an open subset in $\IT$ of finite measure and $\delta>0$. Then
\[
\sum_{R\in\mathscr{M}_1(U)}|R|\,\gamma_1(R)^{-\delta}\lesssim|U|\quad\text{and}\quad\sum_{R\in\mathscr{M}_2(U)}|R|\,\gamma_2(R)^{-\delta}\lesssim|U|.
\]
\end{lemma}

\medskip

\section{Proof of Theorem \ref{single fefferman-Stein}}\label{Single proof}

To keep our exposition self-contained, we recall some basic setup. Due to the presence of the potential $V$, the function $e^{-t\sqrt{\L}}f$ is not classically differentiable, even when  $f$ is  smooth.  Consequently, $\nabla e^{-t\sqrt{\L}}f(x)$ is interpreted in the sense of weak derivatives. 

Let us denote $\n:=(\partial_t,\nabla)$, $\N:=\partial^2_{tt}+\Delta=\n\n$, and $\widetilde{\mathcal{L}}:=-\N+V$.

\subsection{Weak derivatives}\label{sect3}

\begin{lemma}\label{l4} Let $\Omega \subset \IR\times (0,\infty)$ be an open set and \( G \in C^1(\mathbb{R}) \) satisfy \( G(0) = 0 \),  and \( |G'(s)| \leq M\)  for all \(s \in \mathbb{R} \) for some positive constant \( M \). Let \( v \in W^{1,p}_{loc}(\Omega) \) with \( 1 \leq p \leq \infty \). Then
\[ G \circ v \in W^{1,p}(\Omega) \quad \text{and} \quad \n (G \circ v) = (G' \circ v) \n v. \]
\end{lemma}
\begin{proof}
    The proof is similar to that of \cite[Proposition 9.5]{Brezis}.
\end{proof}

\begin{lemma}\label{gradient of uu}
Let $
u(h,t):=e^{-t\sqrt{\L}}f(h)$, where $f\in C_0^\infty(\IR)$.
 Then  $\n u^2 = 2u\n u$  in the weak sense. 
 \end{lemma}
 \begin{proof}
 By \eqref{poisson bound estimate},  there exists a constant $N>0$ such that  $|u(h,t)|\leq N$, for all $h\in \IR$ and $t>0$. Let $G(x)=x^2\phi(x)$, where $\phi(x)$ denotes a smooth function satisfying 
$$
\phi \left( x \right)=\begin{cases}
 	1, \quad &|x|\leqslant N,\\
		0, &|x|\geqslant 2N.\\
\end{cases}
$$
Then $u^2=G(u)$, $G\in C^1(\mathbb{R})$ and $|G^\prime(x)|=|2x\phi(x)+x^2\phi^\prime(x)|\leq M$, where $M$ is a positive constant. By Lemma \ref{l4}, we have 
\begin{align*}
\n (G(u))=\left[2u\phi(u)+u^2\phi^\prime(u)\right]\n u=2u\phi(u)\n u=2u\n u.
\end{align*}
  Hence, $\n u^2=\n (G(u))= 2u\n u$.
\end{proof}

 In order to prove Theorem \ref{single fefferman-Stein}, we require certain estimates for $\{e^{-t\sqrt{\L}}f\}_{t>0}$, and for weak solutions of the equation $\widetilde{\mathcal{L}}u=-\partial^2_{tt}u+ \mathcal{L}u=0$, in domains $\Omega\subset {\mathcal G}\times (0,\infty)$. Define
 $$W^{1,2}_{V}:=\{u\in W^{1,2}(\Omega):\,\int_{\Omega}|u|^2V\,{\ud h \ud t}<\infty\}$$ 
 and let
 $W_{V,0}^{1,2}(\Omega)$ denote the subspace of $W_{V}^{1,2}(\Omega)$ with trace $0$ on $\partial \Omega$. The function  $u\in W_{V}^{1,2}(\Omega) $ is called a weak solution of $\widetilde{\mathcal{L}} u=0$ in $\Omega$ if 
    \begin{align*}
        \int_{\Omega}\n u\n \phi\, {\ud h \ud t}+\int_{\Omega} Vu\phi \,{\ud h \ud t}=0, 
    \end{align*}
    for any $\phi\in W_{V,0}^{1,2}(\Omega)$.

\begin{lemma}\label{The Laplace of u^2}
  Let $
u(h,t)=e^{-t\sqrt{\L}}f(h)$,
where $f\in C_0^\infty(\IR)$.
Then the following equality holds in the weak sense:
    \begin{align*}
        \N{u^2}=2V u^2+2|\n u|^2.
    \end{align*}  
\end{lemma}
\begin{proof}
Let $\phi\in C^\infty_0(\IR\times (0,\infty))$.          
It is easy to see that $u\phi\in W^{1,2}_{V,0}(\IR\times (0,\infty))$ and $\n (u\phi)=u\n \phi +\phi \n u$. Since $u$ is a weak solution of $-\N u+Vu=0$, it follows that
 $$
\int \n u\n (u\phi)\, {\ud h \ud t}+\int Vu^2\phi \,{\ud h \ud t}=0.
$$
This implies 
\begin{align}\label{ee2.2}
\int u\n u\n\phi\, {\ud h \ud t}=-\int |\n u|^2\phi\, {\ud h \ud t}-\int Vu^2\phi \,{\ud h \ud t}.
\end{align}
On the other hand, by  Lemma \ref{gradient of uu}, we have  
   \begin{align}\label{ee2.1}
        \left(\widetilde{\Delta}u^2,\phi \right)&=-\left(\n u^2,\n \phi\right)=-(2u\n u,\n \phi).
        \end{align}
Combining \eqref{ee2.1} and \eqref{ee2.2} yields
$$
\left(\widetilde{\Delta}u^2,\phi \right)=\int 2\left(|\n u|^2+Vu^2\right)\phi \,{\ud h \ud t},
$$
which completes the proof of this lemma.
\end{proof}

\medskip

\subsection{Proof of Theorem \ref{single fefferman-Stein}}\label{section 3.2}
\begin{proof}
 Actually, we will prove a stronger version of Theorem \ref{single fefferman-Stein},  where the area operator $\Ss$ is replaced by $\Sn$, the area operator  associated with the full gradient $\n$,  defined by
\[\Sn(f)(g):=\left(\int_{\Gamma(g)}\left|t\n e^{-t\sqrt{\L}}(f)(h)\right|^2\,\,\cfrac{\ud h \ud t}{t^{Q+1}}\right)^{\frac{1}{2}}.\]

To begin with, we introduce some notation.
\begin{align*}
    E_\lambda=\{g\in\IR:\,N_\L^\beta(f)(g)&\leq \lambda\},\ \ E_{\lambda}^c=\{g\in\IR:\,N_\L^\beta(f)(g)> \lambda\}, \\
A_\lambda=\Big\{g\in\IR:\,\mathcal{M}(\chi_{E_{\lambda}^c})(g)&\leq\frac{1}{10C_0}\Big\},\ \ A_{\lambda}^c=\Big\{g\in\IR:\,\mathcal{M}(\chi_{E_{\lambda}^c})(g)>\frac{1}{10C_0}\Big\},
\end{align*}
where $C_0$ is as  in  \eqref{poisson controlled by H--L}. 
 
 For any $\lambda>0$, we have 
   \begin{align*}
        &\left|\{g\in\IR:\,\,\Sn(f)(g)>\lambda\}\right|\\
        &\qquad\leq \left|\{g\in A_\lambda^c:\,\,\Sn(f)(g)>\lambda\}\right|+\left|\{g\in A_\lambda:\,\,\Sn(f)(g)>\lambda\}\right|\\
        &\qquad\leq |A_\lambda^c|+ \cfrac{1}{\lambda^2}\int_{A_\lambda}\left|\Sn(f)(g)\right|^2\,\ud g,
    \end{align*}
    where the last inequality follows from Chebyshev's inequality.

It follows from  the $L^2$-boundedness of the Hardy--Littlewood maximal function that 
$$|A_\lambda^c|\lesssim |E_\lambda^c|.$$
Thus, to prove Theorem \ref{single fefferman-Stein}, it suffices to prove
\begin{align}
\int_{A_\lambda}\left|\Sn(f)(g)\right|^2\,\ud g\lesssim
\int_{E_\lambda}\left|N_\L^\beta(f)(g)\right|^2\,\ud g+{\lambda^2}|E_\lambda^c|.
\end{align}
Denote $u(h,t):=e^{-t\sqrt{\L}}f(h)$ and $W=\bigcup\limits_{g\in A_\lambda}\Gamma(g)$. We apply Tonelli's theorem to get 
\begin{align}
    \int_{A_\lambda}\left|\Sn(f)(g)\right|^2\,\ud g\lesssim \int_W\left|\n u(h,t)\right|^2t\,\,{\ud h\ud t}.\label{l3}
\end{align}

To proceed, we need to extend the region $W$ to the entire space while preserving the information on $W$.  Let 
\[\widetilde{W^{\beta}}=:\bigcup\limits_{g\in E_\lambda}\Gamma^\beta(g),\]
where  $\Gamma^\beta(g)=\{(h,t)\in\IR\times (0,\infty):\,\,\rho(h,g)<\beta t\}$, and 
$\beta>1$ is  a constant to be chosen later.  We also define a smooth cutoff function $\phi: \mathbb{R} \to \mathbb{R}$ satisfying
\begin{align}\label{cutoff-function}
\phi(t) = \begin{cases}
1, & t \ge \dfrac{9}{10}, \\[1em]
0, & t \le \dfrac{1}{10},
\end{cases}
\end{align}
with $\phi$ smoothly interpolating between $0$ and $1$ on the interval $\big(\frac{1}{10}, \frac{9}{10}\big)$.

Let $p_{t}(f)(h)=: e^{-t\sqrt{-\Delta}}f(h).$  We claim that 
\begin{align}\label{ee-lower bound}
    p_{t}(\chi_{E_\lambda})(h)\geq \frac{9}{10},  \quad {\rm for \  all} \ (h,t)\in W,
\end{align}
and that, for $\beta$ sufficiently large,
\begin{align}\label{<1/10}
    p_{t}(\chi_{E_\lambda})(h)<\frac{1}{10},  \quad {\rm for \  all} \ (h,t)\notin \widetilde{W^{\beta}}.
\end{align}

We first consider \eqref{ee-lower bound}. Since $p_{t}(1)=1$, we have
\begin{align}\label{ee3.1}
    p_t(\chi_{E_\lambda})(h)&=p_{t}(1-\chi_{E_\lambda^c})(h)=1-p_t(\chi_{E_\lambda^c})(h).
\end{align}
 If $(h,t)\in W $,  there exists $z_0\in A_\lambda$ such that 
$\rho(h,z_0)<t$, and thus $z_0\in B(h,t)$.
By the definition of $A_\lambda$, we obtain
\begin{align*}
\frac{\left|E_{\lambda}^c\bigcap B(h,t)\right|}{\left|B(h,t)\right|} \leq \mathcal{M}(\chi_{E_{\lambda}^c})(z_0)\leq \cfrac{1}{10C_0}.    
\end{align*}
Together with \eqref{poisson controlled by H--L}, this gives
 \begin{align*}
 p_{t}(\chi_{E_\lambda^c})(h)\leq C_0\mathcal{M}(\chi_{E_\lambda^c})(z_0)\leq 1/10.
 \end{align*}
Inserting into \eqref{ee3.1} yields \eqref{ee-lower bound}.

Now consider \eqref{<1/10}. If $(h,t)\notin \widetilde{W^{\beta}}$, then $\rho(h,z)\geq\beta t$  for all $z\in E_\lambda$. Hence
  \begin{align*}   
      p_{t}(\chi_{E_\lambda})(h)&\leq C\int_\IR\frac{t}{\left(t+\rho(h,z)\right)^{Q+1}}\chi_{E_\lambda}(z)\,\,{\ud z}\\
      &\leq C\int_{\{z:\,\rho(h,z)\geq \beta t\}}\frac{t}{\left(t+\rho(h,z)\right)^{Q+1}}\,\,{\ud z} \leq \frac{C}{\beta}.
\end{align*}
By choosing $\beta$ large enough so that $C/\beta<1/10$, we obtain \eqref{<1/10}.


We introduce a nonnegative smooth cutoff function $\psi\in C_0^\infty(B(0,1))$ satisfying $\psi\equiv1$ on $B(0,\frac{1}{2})$.
By \eqref{ee-lower bound} and the definition of $\phi$, 
\begin{equation}\label{bbb2}
   \int_W\left|\n u(h,t)\right|^2t\,\,{\ud h\ud t} \leq \lim_{a\to0,b\to\infty}\lim_{R\to\infty}\int_{a}^b\int_{\ID}\left|\n\, u(h,t)\right|^2\phi^2(p_t(\chi_{E_\lambda})(h))\psi\left(\frac{h}{R}\right)t\,\,{\ud h\ud t}.
\end{equation}

By Lemma \ref{The Laplace of u^2} and $V\geq 0$, we have 
\begin{align}\label{cancel-V}
 |\n u|^2\leq |\n u|^2+2Vu^2= \frac{1}{2}\N(u^2).
\end{align}
By  Leibniz's rule,  
\begin{equation}\label{p5}
    \begin{aligned}
    \N(u^2\phi^2)=\n\n \left(u^2\phi^2\right)&=\n\n\left({u^2}\right)\cdot\phi^2+2\n\left({u^2}\right)\cdot\n(\phi^2)+{u^2}\n\n (\phi^2).
\end{aligned}
\end{equation}
We note that here $\n$ and $\N$ are understood in the weak sense.

 \eqref{cancel-V} and \eqref{p5}, together with the fact
 $\N(p_{t}(\chi_{E_\lambda}))
   =0$,  imply   that
 \begin{equation}\label{d1}
\begin{aligned}
&\left|\n\,u(h,t)\right|^2\phi^2(p_t(\chi_{E_\lambda})(h))\p\\
   &\leq \frac{1}{2}\n\,\n\,\left[\left|u(h,t)\right|^2\phi^2(p_t(\chi_{E_\lambda})(h))\right]\p\\
   &\quad-4u(h,t)\n\, u(h,t)\phi(p_t(\chi_{E_\lambda})(h))\phi^\prime(p_t(\chi_{E_\lambda})(h))\n\,(p_t(\chi_{E_\lambda})(h))\p\\
   &\quad-\left|u(h,t)\right|^2\phi(p_t(\chi_{E_\lambda})(h))\phi^{\prime\prime}(p_t(\chi_{E_\lambda})(h))\left|\n\,\left(p_t(\chi_{E_\lambda})(h)\right)\right|^2\p\\
   &\quad-\left|u(h,t)\right|^2\left|\phi^{\prime}(p_t(\chi_{E_\lambda})(h))\right|^2\left|\n\left(p_t(\chi_{E_\lambda})(h)\right)\right|^2\p\\
   &=: f_1(h,t)+f_2(h,t)+f_3(h,t)+f_4(h,t).
\end{aligned}
\end{equation}
By the Cauchy--Schwarz inequality, we have
\begin{align*}
    \left|f_2(h,t)\right|&\leq \frac{1}{2}\left|\n\,u(h,t)\right|^2\left|\phi(p_t(\chi_{E_\lambda})(h))\right|^2\p\\
    &\qquad+32\left|u(h,t)\right|^2\left|\phi^{\prime}(p_t(\chi_{E_\lambda})(h))\right|^2\left|\n\,p_t(\chi_{E_\lambda})(h)\right|^2\p\\
    &=: f_{21}(h,t)+f_{22}(h,t).
\end{align*}

We note that $f_{21}(h,t)$ is absorbed by the left hand of \eqref{d1}, while $f_{22}(h,t)$ is similar to $f_4(h,t)$, hence, we only need to consider $f_1(h,t)$, $f_3(h,t)$ and $f_4(h,t)$.

Then, \eqref{bbb2} is controlled by 
\begin{align*}
    &\left|\lim_{a\to 0,b\to\infty}\lim_{R\to\infty}\int_a^b\int_{\IR} f_1(h,t) t\,{\ud h\ud t}\right|+\left|\lim_{a\to 0,b\to\infty}\lim_{R\to\infty}\int_a^b\int_{\IR} f_3(h,t) t\,{\ud h\ud t}\right|\\
    &+\left|\lim_{a\to 0,b\to\infty}\lim_{R\to\infty}\int_a^b\int_{\IR} f_4(h,t) t\,\ud h\ud t\right|\\
    &=: \frak{I}_1+\frak{I}_2+\frak{I}_3.
\end{align*}

For $\frak{I}_3$, we have
\begin{align*}
    \frak{I}_3&\leq\int_0^\infty\int_{\IR}\left|u(h,t)\right|^2\left|\phi^{\prime}(p_t(\chi_{E_\lambda})(h))\right|^2\left|\n\left(p_t(\chi_{E_\lambda})(h)\right)\right|^2t\,\ud h \ud t.
\end{align*}

From the definitions of  $\phi$, $\phi^\prime(p_t(\chi_{E_\lambda})(h))\neq 0$ implies $p_t(\chi_{E_\lambda})(h)\geq \frac{1}{10}$. Consequently, $(h,t)\in \widetilde{W^\beta}$. Then, there exists $g_1\in E_\lambda$ such that $\rho(h,g_1)<\beta t$ and $|u(h,t)|\leq N^\beta_{\L} f(g_1)\leq \lambda$. Therefore, 
\begin{align*}
    \frak{I}_3&\leq \int_{\mathcal{G}\times\mathbb{R}_+} \lambda^2\left|\n\,(p_t(\chi_{E_\lambda})(h))\right|^2 t\,\ud h\ud t\\
    &\leq C\lambda^2\int_{\mathcal{G}\times\mathbb{R}_+} \left|t\n\,(p_t(1-\chi_{E_\lambda})(h))\right|^2 \,\cfrac{\ud h \ud t}{t}\\
    &\leq C \lambda^2 \|(1-\chi_{E_\lambda})\|_{L^2}^2\\
    &\leq C\lambda^2|E_\lambda^c|,
\end{align*}
where  in the third inequality we have used the $L^2$-boundedness of the Littlewood--Paley square function.

Following the same argument as  $\frak{I}_3$, one can obtain
\[
\frak{I}_2\leq C\lambda^2|E_\lambda^c|.
\]



It remains to estimate $\frak{I}_1$.  
\begin{align*}
\frak{I}_1&\lesssim\bigg|\lim_{a\to0,b\to\infty}\lim_{R\to\infty}\int_a^b\int_{\IR}\partial_{tt}\left[\left|u(h,t)\right|^2\phi^2(p_t(\chi_{E_\lambda})(h))\right]t\p\,\ud h\ud t\bigg| \\ 
&\quad\quad\quad+\bigg|\lim_{a\to0,b\to\infty}\lim_{R\to\infty}\int_a^b\int_{\IR}\Delta_h\left[\left|u(h,t)\right|^2\phi^2(p_t(\chi_{E_\lambda})(h))\right]t\p\,\ud h\ud t\bigg|\\
&=:|\frak{I}_{11}|+|\frak{I}_{12}|.
\end{align*}
By Fubini's theorem, integration by parts, and Lebesgue's dominated convergence theorem,  we have
\begin{align*}
\frak{I}_{11}&=\lim_{a\to0,b\to \infty}\lim_{R\to\infty}\int_{\IR} \partial_{t}\left[\left|u(h,t)\right|^2\phi^2(p_t(\chi_{E_\lambda})(h))\right]\p t{\bigg|_{t=a}^{b}}\,\ud h \\
&\qquad-\lim_{a\to0,b\to \infty}\lim_{R\to\infty}\int_{\IR}\int_{a}^{b}\partial_t\left[\left|u(h,t)\right|^2\phi^2(p_t(\chi_{E_\lambda})(h))\right]\p\,\ud t\ud h\\
&=\int_\IR \bigg\{\partial_{t}\left[\left|u(h,t)\right|^2\phi^2(p_t(\chi_{E_\lambda})(h))\right]t\bigg\}{\bigg|_{t=0}^{\infty}}\,\ud h\\
&\qquad-\lim_{a\to 0,b\to \infty}\lim_{R\to\infty}\int_{\IR}\left[\left|u(h,t)\right|^2\phi^2(p_t(\chi_{E_\lambda})(h))\right]\bigg|_{t=a}^{b}\p\,\ud h\\
&=\int_\IR \bigg\{\partial_{t}\left[\left|u(h,t)\right|^2\phi^2(p_t(\chi_{E_\lambda})(h))\right]t\bigg\}{\bigg|_{t=0}^{\infty}}\,\ud h-\int_{\IR}\left[\left|u(h,t)\right|^2\phi^2(p_t(\chi_{E_\lambda})(h))\right]\bigg|_{t=0}^{\infty}\,\ud h.
\end{align*}

\noindent Note that 
\begin{align}\label{single 3.28}
&\partial_{t}\left[\left|u(h,t)\right|^2\phi^2(p_t(\chi_{E_\lambda})(h))\right]t\notag\\
&=2u(h,t)t\partial_{t}u(h,t)\phi^2(p_t(\chi_{E_\lambda})(h))\\
&\quad+2u(h,t)^2\phi(p_t(\chi_{E_\lambda})(h))\phi^\prime(p_t(\chi_{E_\lambda})(h))t\partial_{t}p_t(\chi_{E_\lambda})(h)\notag,
\end{align}
and 
\begin{align}
&\lim_{t\to \infty}u(h,t)=0, \quad  \lim_{t\to \infty}t\partial_tu(h,t)=0, \quad \lim_{t\to \infty}t\partial_tp_t(\chi_{E_\lambda})(h,t)=0,\nonumber\\
& \lim_{t\to 0}u(h,t)=f(h),\quad \lim_{t\to 0}t\partial_tu(h,t)=0, \quad \lim_{t\to 0}t\partial_tp_t(\chi_{E_\lambda})(h,t)=0.\label{single 3.29}
\end{align}

Then, we have  
$$\int_\IR \bigg\{\partial_{t}\left[\left|u(h,t)\right|^2\phi^2(p_t(\chi_{E_\lambda})(h))\right]t\bigg\}{\bigg|_{t=0}^\infty}\,\ud h=0
$$
and 
\begin{align*}
\Bigg|\int_{\IR}\left[\left|u(h,t)\right|^2\phi^2(p_t(\chi_{E_\lambda})(h))\right]\bigg|_{t=0}^{\infty}\,\ud h\Bigg|&= \int_\IR \left|f(h)\right|^2\phi^2(\chi_{E_\lambda}(h))\,\ud h,\\
    &\leq \int_{E_\lambda} (N^\beta_{\L}(f)(h))^2\,\ud h.
\end{align*}
These estimates imply that
$$
|\frak{I}_{11}|\lesssim  \int_{E_\lambda} (N^\beta_{\L}(f)(h))^2\,\ud h.
$$

For $\frak{I}_{12}$, integrating by parts in $h$, we obtain that for any $0<a<b$,
\begin{align*}
&\left|\lim_{R\to\infty}\int_a^b\int_{\IR}\Delta_h\left[\left|u(h,t)\right|^2\phi^2(p_t(\chi_{E_\lambda})(h))\right]t\p\,\ud h\ud t\right|\\   
   &=\left|\lim_{R\to \infty}\frac{1}{R^2}\int_a^b\int_{\IR}(\Delta_h\psi)\left(\frac{h}{R}\right)\left[\left|u(h,t)\right|^2\phi^2(p_t(\chi_{E_\lambda})(h))\right]t\ud h\ud t\right|\\
&\lesssim\lim_{R\to \infty}\frac{b}{R^2}\int_a^b\int_{\IR}\left|u(h,t)\right|^2\ud h\ud t=0,
\end{align*}
which implies that
$\frak{I}_{12}=0$, and thus $\frak{I}_{1}=0$. This completes the proof.
\end{proof}
\medskip
\bigskip

\section{Proof of Theorem \ref{prod F-S thm}}\label{prouct proof}
\begin{proof}
We will prove a stronger version of Theorem \ref{prod F-S thm},  
where the area operator $\SP(f)$  is replaced by $\Snp (f)$, the area operator  associated with the full gradient,  defined by
\[\Snp(f)(\g):=\left(\int_{\Gamma(\g)}\left|t_1\n_1 p_{t_1}^{\L_1} t_2\n_2 p_{t_2}^{\L_2}(f)(\h)\right|^2\,\,\cfrac{\ud \h\,\t}{t_1^{Q_1+1}t_2^{Q_2+1}}\right)^{\frac{1}{2}}.\]
 Here for $i=1,2$,  $\n_i$ stands for the  full gradient on $\mathcal{G}_i$, i.e., $(\partial_{t_i},\nabla_{g_i})$, and
\[
p_{t_i}^{\L_i}f:= e^{-t_i\sqrt{\L_i}}f.
\]
We also denote $p_{t_i}f:= e^{-t_i\sqrt{-\Delta_i}}f$, for $i=1,2$.

Next, for $\lambda>0$ and $f\in L^1(\IT)$ satisfying $N^\beta_{\L_1,\L_2}(f)\in L^1(\IT)$, define
\begin{align*}
&E_\lambda=\{\g\in\IT:\,N^\beta_{\L_1,\L_2}(f)(\g)\leq \lambda\},\quad E_{\lambda}^c=\{\g\in\IT:\,N^\beta_{\L_1,\L_2}(f)(\g)> \lambda\},\\
&A_\lambda=\Big\{\g\in\IT:\,\mathcal{M}_s(\chi_{E_{\lambda}^c})(\g)\leq \frac{1}{10C_1}\Big\},\,\,A_{\lambda}^c=\Big\{\g\in\IT:\,\mathcal{M}_s(\chi_{E_{\lambda}^c})(\g)>\frac{1}{10C_1}\Big\},
\end{align*}
where $C_1$ is as in \eqref{product poisson controlled}.

Similar to Theorem \ref{single fefferman-Stein}, the key step in this proof is to   prove that there exist constants $C>0$ and $\beta>1$ such that for all $f\in C_0^\infty(\IT)$
 and $\lambda>0$, 
\begin{align}\label{good lambda}
 \int_{A_\lambda}\left|\Snp(f)(\h)\right|^2\,\,{\rm d\h}\lesssim
    \int_{E_\lambda}\left|N^\beta_{\L_1,\L_2}(f)(\h)\right|^2\,\,{{\bf {\rm d\h}}}+\lambda^2|E_\lambda^c|.
\end{align}

We now  prove \eqref{good lambda}.
It follows from  Tonelli's theorem that
\begin{align}
    \int_{A_\lambda}\left|\Snp(f)(\h)\right|^2\,\,{{\bf {\rm d\h}}}\lesssim \int_W\left|\n_1 p_{t_1}^{\L_1} u_{h_2,t_2}(h_1)\right|^2t_1t_2\,\,\ud \h\,\t, \label{ll3}
\end{align}
where $W=\bigcup\limits_{\g\in A_\lambda}\Gamma(\g)$ and $u_{h_2,t_2}(h_1)=:\n_2p_{t_2}^{\L_2}(f)(h_1,h_2)$.

To proceed, we will extend the region $W$ to the whole space while preserving the information on $W$. Define
\[\widetilde{W^{\beta}}:=\bigcup\limits_{\g\in E_\lambda}\Gamma^\beta(\g)
\quad {\rm and}\quad
m_{\lambda}(\g):= \chi_{E_\lambda}(g_1,g_2),
\]
where $\Gamma^\beta(\g)=\Gamma_1^\beta(g_1)\otimes\Gamma_2^\beta(g_2)$ with $\beta>1$.  Let $\phi$ be given by \eqref{cutoff-function}.

As in \eqref{ee-lower bound} and \eqref{<1/10}, if $(\h, {\bf t}) \in W$,  then the definition of $A_\lambda$, \eqref{product poisson controlled} and the product Poisson kernel estimates yield
 \begin{align}\label{prod 4.33}
p_{t_1} p_{t_2}(m_\lambda)(\h) \geq \frac{9}{10}.    
\end{align}
 If $(\h, {\bf t}) \notin \widetilde{W^{\beta}}$, we choose  $\beta > 1$ sufficiently large so that 
 \begin{align}\label{prod 4.34}
 p_{t_1} p_{t_2}(m_\lambda)(\h) \leq \frac{C}{\beta} < \frac{1}{10}.
 \end{align}




Now we take  smooth functions $\psi_i$ on $\IR_i$ such that $\psi_i=1$ on $B(0,1)$ and $\psi_i\in C_0^\infty(B(0,2))$ for $i=1,2.$
Hence, by \eqref{prod 4.33} and the definition of $\phi$, the right term of \eqref{ll3} is controlled by 
\begin{align}\label{b2}
    \ir&\left|\n_1 p_{t_1}^{\L_1} u_{h_2,t_2}(h_1)\right|^2\phi^2(p_{t_1}p_{t_2}(m_{\lambda})(\h))\\
    &\quad\quad\times\psi_1(h_1/R_1)\psi_2(h_2/R_2) t_1t_2\,\,\ud \h\,\t.\notag
\end{align} 
Denote $\widetilde{\psi}_{\bf R}(\h):=\psi_1(h_1/R_1)\psi_2(h_2/R_2)$. Noting that $p_{t_1}p_{t_2}(m_{\lambda})(\g)$, as a function of $(g_1,t_1)$, is harmonic on $\G\times\mathbb{R}_+$,  the same argument as in \eqref{d1} gives
\begin{flalign}
&\left|\n_1p_{t_1}^{\L_1}u_{h_2,t_2}(h_1)\right|^2\phi^2(p_{t_1}p_{t_2}(m_{\lambda})(\h))\widetilde{\psi}_{\bf R}(\h) & \label{dd1}\\
&\leq \frac{1}{2}\n_1\n_1\left[\left|p_{t_1}^{\L_1}u_{h_2,t_2}(h_1)\right|^2\phi^2(p_{t_1}p_{t_2}(m_{\lambda})(\h))\right]\widetilde{\psi}_{\bf R}(\h) & \notag\\
&\quad -4p_{t_1}^{\L_1}u_{h_2,t_2}(h_1)\n_1 p_{t_1}^{\L_1}u_{h_2,t_2}(h_1)  \phi(p_{t_1}p_{t_2}(m_{\lambda})(\h))\phi^\prime(p_{t_1}p_{t_2}(m_{\lambda})(\h))\n_1(p_{t_1}p_{t_2}(m_{\lambda})(\h))\widetilde{\psi}_{\bf R}(\h)& \notag\\
&\quad-\left|p_{t_1}^{\L_1}u_{h_2,t_2}(h_1)\right|^2\phi(p_{t_1}p_{t_2}(m_{\lambda})(\h))\phi^{\prime\prime}(p_{t_1}p_{t_2}(m_{\lambda})(\h)) \left|\n_1(p_{t_1}p_{t_2}(m_{\lambda})(\h))\right|^2\widetilde{\psi}_{\bf R}(\h) & \notag\\
&\quad-\left|p_{t_1}^{\L_1}u_{h_2,t_2}(h_1)\right|^2\left|\phi^{\prime}(p_{t_1}p_{t_2}(m_{\lambda})(\h))\right|^2\left|\n_1p_{t_1}p_{t_2}(m_{\lambda})(\h)\right|^2\widetilde{\psi}_{\bf R}(\h) & \notag\\
&=: \sum_{j=1}^{4}f_j(h_1,h_2,t_1,t_2) .\notag
\end{flalign}

By the Cauchy--Schwarz inequality, we have
\begin{equation*}
    \begin{aligned}
   & \left|f_2(h_1,h_2,t_1,t_2)\right|\\
    &\leq \frac{1}{2}\left|\n_1p_{t_1}^{\L_1}u_{h_2,t_2}(h_1)\right|^2\left|\phi(p_{t_1}p_{t_2}(m_{\lambda})(\h))\right|^2\psi_1(h_1/R_1)\psi_2(h_2/R_2)\\
    &\quad+32\left|p_{t_1}^{\L_1}u_{h_2,t_2}(h_1)\right|^2\left|\phi^{\prime}(p_{t_1}p_{t_2}(m_{\lambda})(\h))\right|^2\left|\n_1p_{t_1}p_{t_2}(m_{\lambda})(\h)\right|^2\psi_1(h_1/R_1)\psi_2(h_2/R_2)\\
    &=: f_{21}(h_1,h_2,t_1,t_2)+f_{22}(h_1,h_2,t_1,t_2).
\end{aligned}
\end{equation*}
Note that $f_{21}$ can be absorbed into the left-hand side of \eqref{dd1}, while $f_{22}$ is similar to  $f_4$ and $f_3$. One has
\begin{align*}
    &f_{22}(h_1,h_2,t_1,t_2)+f_3(h_1,h_2,t_1,t_2)+f_4(h_1,h_2,t_1,t_2)\\
    &\lesssim \left|p_{t_1}^{\L_1}u_{h_2,t_2}(h_1)\right|^2\left|\Phi^{\prime}(p_{t_1}p_{t_2}(m_{\lambda})(\h))\right|^2\left|\n_1p_{t_1}p_{t_2}(m_{\lambda})(\h)\right|^2\psi_1(h_1/R_1)\psi_2(h_2/R_2)\\
    &=: F_2(h_1,h_2,t_1,t_2),
\end{align*}
where we choose a smooth function $\Phi(t)$ on ${\mathbb R}$ such that 
\begin{align*}
    \Phi^{\prime}(t) = \left(\phi^{\prime}(t)^4+[\phi(t)\phi^{\prime\prime}(t)]^2\right)^{\frac{1}{4}} \quad  \text{and}\quad \Phi(\frac{1}{10})=0.
\end{align*}
Note that $\Phi^\prime(t)=0$ for $t<\frac{1}{10}$ or $t>1$, and $\Phi$  has similar properties to $\phi$.

Then \eqref{b2} is controlled by 
\begin{align*}
    &2\left|\ir f_1(h_1,h_2,t_1,t_2) t_1t_2\,\ud \h\,\t\right|\\
    &\quad+2\left|\ir F_2(h_1,h_2,t_1,t_2) t_1t_2\,\ud \h\,\t\right|\\
    &=: 2\frak{I}_1+2\frak{I}_2.
\end{align*}
For $\frak{I}_1$, we have
\begin{align*}
   \frak{I}_1 &\leq \frac{1}{2}\bigg|\ir \partial_{t_1}^2\left[\left|p_{t_1}^{\L_1}u_{h_2,t_2}(h_1)\right|^2\phi^2(p_{t_1}p_{t_2}(m_{\lambda})(\h))\right]\\
    &\hskip4.8cm\times t_1t_2\psi_1 (h_1/R_1)\psi_2(h_2/R_2)\,\,\ud \h\,\t \bigg|\\
    &\quad+\frac{1}{2}\bigg|\ir\Delta_{h_1}\left[\left|p_{t_1}^{\L_1}u_{h_2,t_2}(h_1)\right|^2\phi^2(p_{t_1}p_{t_2}(m_{\lambda})(\h))\right]\\
    &\hskip5.3cm
    \times t_1t_2\psi_1(h_1/R_1)\psi_2(h_2/R_2)\,\,{\ud \h\,\t}\bigg|\\
    &=:|\frak{I}_{11}|+|\frak{I}_{12}|.
\end{align*}

For $\frak{I}_{11}$, by Fubini's theorem, integration by parts in $t_1$ and Lebesgue's dominated convergence theorem, we have
\begin{align}
\frak{I}_{11}
    &=\lim_{a_2\to0\atop b_2\to\infty}\lim_{R_2\to\infty}\lim_{a_1\to0\atop b_1\to\infty}\lim_{R_1\to\infty}\int_{a_2}^{b_2}\int_{\IR_2}\int_{\IR_1}\partial_{t_1}\left[\left|p_{t_1}^{\L_1}u_{h_2,t_2}(h_1)\right|^2\phi^2(p_{t_1}p_{t_2}(m_{\lambda})(\h))\right]\notag\\
    &\hskip3cm \times t_1t_2\psi_1(h_1/R_1)\psi_2(h_2/R_2)\bigg|_{t_1=a_1}^{b_1}\,\ud h_1\ud h_2\ud t_2\notag\\
    &\quad-\lim_{a_2\to0\atop b_2\to\infty}\lim_{R_2\to\infty}\lim_{a_1\to0\atop b_1\to\infty}\lim_{R_1\to\infty}\int_{a_2}^{b_2}\int_{\IR_2}\int_{a_1}^{b_1}\int_{\IR_1}\partial_{t_1}\left[\left|p_{t_1}^{\L_1}u_{h_2,t_2}(h_1)\right|^2\phi^2(p_{t_1}p_{t_2}(m_{\lambda})(\h))\right]\notag\\
    &\hskip3cm\times t_2\psi_1(h_1/R_1)\psi_2(h_2/R_2)\,\ud \h\,\t\notag\\
    &=\lim_{a_2\to0\atop b_2\to\infty}\lim_{R_2\to\infty}\lim_{R_1\to\infty}\int_{a_2}^{b_2}\int_{\IR_2}\int_{\IR_1}\partial_{t_1}\left[\left|p_{t_1}^{\L_1}u_{h_2,t_2}(h_1)\right|^2\phi^2(p_{t_1}p_{t_2}(m_{\lambda})(\h))\right]\label{I11}\\
    &\hskip3cm \times t_1t_2\psi_1(h_1/R_1)\psi_2(h_2/R_2)\bigg|_{t_1=0}^{\infty}\,\ud h_1\ud h_2\ud t_2\notag\\
    &\quad-\lim_{a_2\to0\atop b_2\to\infty}\lim_{R_2\to\infty}\lim_{R_1\to\infty}\int_{a_2}^{b_2}\int_{\IR_2}\int_{\IR_1}\left[\left|p_{t_1}^{\L_1}u_{h_2,t_2}(h_1)\right|^2\phi^2(p_{t_1}p_{t_2}(m_{\lambda})(\h))\right]\notag\\
    &\hskip3cm\times t_2\psi_1(h_1/R_1)\psi_2(h_2/R_2)\bigg|_{t_1=0}^{\infty}\ud h_1 \ud h_2\ud t_2\notag\\
    &=:\frak{I}_{11}^\prime+\frak{I}^{\prime\prime}_{11}.\notag
\end{align}

\noindent As in \eqref{single 3.28} and \eqref{single 3.29}, we obtain $\frak{I}_{11}^\prime=0$ and
\begin{align*}
    \frak{I}_{11}^{\prime\prime}
    =\lim_{a_2\to 0\atop b_2\to \infty}\lim_{R_2\to\infty}\int_{a_2}^{b_2}\int_{\IR_2}\int_{\IR_1}  \left[\left|u_{h_2,t_2}(h_1)\right|^2\phi^2(p_{t_2}(m_{\lambda})(\h))\right]\psi_2(h_2/R_2)t_2\ud h_1\ud  h_2\ud t_2.
\end{align*}

For $\frak{I}_{12}$, by integration by parts, we have
\begin{align}
&\int_{\IR_1}\Delta_{h_1}\left[\left|p_{t_1}^{\L_1}u_{h_2,t_2}(h_1)\right|^2\phi^2(p_{t_1}p_{t_2}(m_{\lambda})(\h))\right]t_1t_2\psi_1(h_1/R_1)\psi_2(h_2/R_2)\ud h_1\notag\\
&=\int_{\IR_1} \left[\left|p_{t_1}^{\L_1}u_{h_2,t_2}(h_1)\right|^2\phi^2(p_{t_1}p_{t_2}(m_{\lambda})(\h))\right]t_1t_2\Delta_{h_1}(\psi_1(h_1/R_1))\psi_2(h_2/R_2)\ud h_1\label{J_12}\\
&\leq \cfrac{1}{R_1^2}\int_{\IR_1}  \left[\left|p_{t_1}^{\L_1}u_{h_2,t_2}(h_1)\right|^2\phi^2(p_{t_1}p_{t_2}(m_{\lambda})(\h))\right]t_1t_2\psi_2(h_2/R_2)\ud h_1\to 0, \notag
\end{align}
as $R_1\to \infty$. This yields $\frak{I}_{12}=0$.

Therefore, 
\begin{align*}
    \frak{I}_1&\leq\frac{1}{2}\lim_{a_2\to 0\atop b_2\to \infty}\lim_{R_2\to\infty}\int_{a_2}^{b_2}\int_{\IR_2}\int_{\IR_1}  \left[\left|u_{h_2,t_2}(h_1)\right|^2\phi^2(p_{t_2}(m_{\lambda})(\h))\right]\psi_2(h_2/R_2)t_2\ud\h\ud t_2\\
    &=\frac{1}{2}\lim_{a_2\to 0\atop b_2\to \infty}\lim_{R_2\to\infty}\int_{\IR_1}\int_{a_2}^{b_2}\int_{\IR_2} \left[\left|\n_2p_{t_2}^{\L_2}(f)(\h)\right|^2\phi^2(p_{t_2}(m_{\lambda})(\h))\right]\psi_2(h_2/R_2)t_2\ud h_2\ud t_2\ud h_1.
\end{align*}

Similar to \eqref{d1}, we have
\begin{align}
    &\left|\n_2p_{t_2}^{\L_2}(f)(\h)\right|^2\phi^2(p_{t_2}(m_{\lambda})(\h))\psi_2(h_2/R_2)\label{d2}\\
    &\leq \frac{1}{2}\n_2\n_2\left[\left|p_{t_2}^{\L_2}(f)(\h)\right|^2\phi^2(p_{t_2}(m_{\lambda})(\h))\right]\psi_2(h_2/R_2)\notag\\
    &\quad -4p_{t_2}^{\L_2}(f)(\h)\n_2p_{t_2}^{\L_2}(f)(\h)\phi(p_{t_2}(m_{\lambda})(\h))\phi^\prime(p_{t_2}(m_{\lambda})(\h))\n_2(p_{t_2}(m_{\lambda})(\h))\psi_2(h_2/R_2)\notag\\
    &\quad-\left|p_{t_2}^{\L_2}(f)(\h)\right|^2\phi(p_{t_2}(m_{\lambda})(\h))\phi^{\prime\prime}(p_{t_2}(m_{\lambda})(\h))\left|\n_2p_{t_2}(m_{\lambda})(\h)\right|^2\psi_2(h_2/R_2)\notag\\
    &\quad -\left|p_{t_2}^{\L_2}(f)(\h)\right|^2\phi^{\prime}(p_{t_2}(m_{\lambda})(\h))\phi^{\prime}(p_{t_2}(m_{\lambda})(\h))\left|\n_2p_{t_2}(m_{\lambda})(\h)\right|^2\psi_2(h_2/R_2)\notag\\
    &=: \sum_{j=1}^{4}\mathcal{H}_j(h_1,h_2,t_1,t_2).\notag
\end{align}

For $\mathcal{H}_2$, using the Cauchy--Schwarz inequality,  $\mathcal{H}_2$ is controlled by 
\begin{align*}
    &\frac{1}{2}\left|\n_2p_{t_2}^{\L_2}(f)(\h)\right|^2\left|\phi^2(p_{t_2}(m_{\lambda})(\h))\right|^2\psi_2(h_2/R_2)\\
    &\quad+32\left|p_{t_2}^{\L_2}(f)(\h)\right|^2\left|\phi^\prime(p_{t_2}(m_{\lambda})(\h))\right|^2\left|\n_2(p_{t_2}(m_{\lambda})(\h))\right|^2\psi_2(h_2/R_2)\\
    &=:  {\mathcal H}_{21}(h_1,h_2,t_1,t_2)+{\mathcal H}_{22}(h_1,h_2,t_1,t_2).
\end{align*}
Note that $\mathcal{H}_{21}$ can be absorbed into the left-hand side of \eqref{d2} and $\mathcal{H}_{22}$ is similar to  $\mathcal{H}_3$ and $\mathcal{H}_4$.  One has
\begin{align*}
    &\mathcal{H}_{22}(h_1,h_2,t_1,t_2)+\mathcal{H}_3(h_1,h_2,t_1,t_2)+\mathcal{H}_4(h_1,h_2,t_1,t_2)\\
    &\lesssim \left|p_{t_2}^{\L_2}(f)(\h)\right|^2\left|\Psi^\prime(p_{t_2}(m_{\lambda})(\h))\right|^2\left|\n_2(p_{t_2}(m_{\lambda})(\h))\right|^2\psi_2(h_2/R_2)\\
    &=: \frak{J}_2(h_1,h_2,t_1,t_2),
\end{align*}
where we choose a smooth function $\Psi(t)$ on ${\mathbb R}$ such that
\begin{align*}
    \Psi^\prime(t)=\left(\phi^{\prime}(t)^4+[\phi\phi^{\prime\prime}(t)]^2  \right)^\frac{1}{4} \quad \text{and} \quad \Psi(\frac{1}{10})=0.
\end{align*}
Note that $\Psi^\prime(t)=0$ for $t<\frac{1}{10}$ or $t>1$, and $\Psi$  has similar properties to $\phi$.
 
We then have
 \begin{align*}
     \frak{I}_1&\leq\bigg|\frac{1}{2}\lim_{a_2\to 0\atop b_2\to \infty}\lim_{R_2\to\infty}\int_{\IR_1}\int_{a_2}^{b_2}\int_{\IR_2} \mathcal{H}_1(h_1,h_2,t_1,t_2) t_2\,\,{\ud h_2\ud t_2\ud h_1}\bigg|\\
     &\quad+\bigg|\frac{1}{2}\lim_{a_2\to 0\atop b_2\to \infty}\lim_{R_2\to\infty}\int_{\IR_1}\int_{a_2}^{b_2}\int_{\IR_2} \frak{J}_2(h_1,h_2,t_1,t_2)t_2\,\,{\ud h_2\ud t_2\ud h_1}\bigg|\\
     &=: |\widetilde{\frak{I}_{11}}|+|\widetilde{\frak{I}_{12}}|.
 \end{align*}

For $\widetilde{\frak{I}_{11}}$, we have
\begin{align*}
    \widetilde{\frak{I}_{11}}&=\frac{1}{2}\lim_{a_2\to 0\atop b_2\to \infty}\lim_{R_2\to\infty}\int_{\IR_1}\int_{a_2}^{b_2}\int_{\IR_2}\partial_{t_2}^2 \left[\left|p_{t_2}^{\L_2}(f)(\h)\right|^2\phi^2(p_{t_2}(m_{\lambda})(\h))\right]\psi_2(h_2/R_2)t_2\ud h_2\ud t_2\ud h_1\\
    &\quad+\frac{1}{2}\lim_{a_2\to 0\atop b_2\to \infty}\lim_{R_2\to\infty}\int_{\IR_1}\int_{a_2}^{b_2}\int_{\IR_2}\Delta_{h_2} \left[\left|p_{t_2}^{\L_2}(f)(\h)\right|^2\phi^2(p_{t_2}(m_{\lambda})(\h))\right]\psi_2(h_2/R_2)t_2\ud h_2\ud t_2\ud h_1\\
    &=:\widetilde{\frak{I}_{111}}+\widetilde{\frak{I}_{112}}.
\end{align*}
Similar to \eqref{J_12}, we have $\widetilde{\frak{I}_{112}}=0.$

\noindent For $\widetilde{\frak{I}_{111}}$,  integrating by parts with respect to $t_2$ yields
\begin{align*}
    \widetilde{\frak{I}_{111}}&=\int_{\IR_1}\int_{\IR_2} \partial_{t_2}\left[\left|p_{t_2}^{\L_2}(f)(\h)\right|^2\phi^2(p_{t_2}(m_{\lambda})(\h))\right]t_2\bigg|_{t_2=0}^{\infty}\,\,{\ud h_2\ud h_1}\\
    &\quad-\int_{\IR_1}\int_{\mathcal{G}_2\times\mathbb{R}_+} \partial_{t_2}\left[\left|p_{t_2}^{\L_2}(f)(\h)\right|^2\phi^2(p_{t_2}(m_{\lambda})(\h))\right]\,\,\ud h_2\ud t_2\ud h_1\\
    &=\int_{\IR_1}\int_{\IR_2} \partial_{t_2}\left[\left|p_{t_2}^{\L_2}(f)(\h)\right|^2\phi^2(p_{t_2}(m_{\lambda})(\h))\right]t_2\bigg|_{t_2=0}^{\infty}\,\,{\ud h_2\ud h_1}\\
    &\quad-\int_{\IR_1}\int_{\IR_2} \left[\left|p_{t_2}^{\L_2}(f)(\h)\right|^2\phi^2(p_{t_2}(m_{\lambda})(\h))\right]\bigg|_{t_2=0}^{\infty}\,\,\ud h_2\ud h_1.
\end{align*}
Similar to \eqref{single 3.28} and \eqref{single 3.29}, one has 
$$\int_{\IR_1}\int_{\IR_2 }\partial_{t_2}\left[\left|p_{t_2}^{\L_2}(f)(\h)\right|^2\phi^2(p_{t_2}(m_{\lambda})(\h))\right]t_2\bigg|_{t_2=0}^{\infty}\ud h_2\ud h_1=0.$$ 
Since 
\[
\lim_{t_2\to\infty}\phi(p_{t_2}(m_{\lambda})(\h))=0,\quad \lim_{t_2\to 0}p_{t_2}^{\L_2}(f)(\h)=f(\h),\quad\lim_{t_2\to 0}p_{t_2}(m_\lambda)(\h)=m_\lambda(\h),
\]
we  have
\begin{align*}
    |\widetilde{\frak{I}_{111}}|\leq \int_{\IR_1}\int_{\IR_2} f^2(\h)\phi^2(m_\lambda(\h))\,\,{\rm d\h}\leq C \int_{E_\lambda} N^\beta_{\L_1,\L_2}(f)(\h)^2\,\,{\rm d\h},
\end{align*}
where in the last inequality we used $|f(\h)|\leq N^\beta_{\L_1,\L_2}(f)(\h)$ and that  the condition $\phi(m_\lambda(\h))\neq 0$ implies $\h\in E_\lambda$. Thus, we have obtained
\begin{align*}
    |\widetilde{\frak{I}_{11}}|\leq |\widetilde{\frak{I}_{111}}|+ |\widetilde{\frak{I}_{112}}|\leq C \int_{E_\lambda} N^\beta_{\L_1,\L_2}(f)(\h)^2\,\,{\rm d\h}.
\end{align*}

For $\widetilde{\frak{I}_{12}}$, we have
\begin{align*}
 |\widetilde{\frak{I}_{12}}|
&\lesssim \int_{\IR_1}\int_{\mathcal{G}_2\times\mathbb{R}_+} \left|p_{t_2}^{\L_2}(f)(\h)\right|^2\left|\Psi^\prime(p_{t_2}(m_{\lambda})(\h))\right|^2\left|\n_2(p_{t_2}(m_{\lambda})(\h))\right|^2t_2\ud h_2\ud t_2\ud h_1.
\end{align*}
By the definitions of $m_\lambda$ and $\Psi$, if $\Psi^{\prime}(p_{t_2}(m_{\lambda})(\h))\neq 0$, then $p_{t_2}(m_\lambda)(\h)>1/10$. From \eqref{prod 4.34}, we have $(h_1,h_2,0,t_2)\in \widetilde{W^\beta}$. Consequently, there exists $g_2\in {\mathcal G}_2$, such that $(h_1,g_2)\in E_\lambda$ and $\left|p_{t_2}^{\L_2}(f)(\h)\right|\leq N^\beta_{\L_1,\L_2}(f)(h_1,g_2)\leq \lambda$. Thus 
\begin{align*}
    |\widetilde{\frak{I}_{12}}|
    &\lesssim \lambda^2\int_{\IR_1}\int_{\II} \left|t_2\n_2(p_{t_2}(m_{\lambda})(\h))\right|^2\,\cfrac{\ud h_2\ud t_2\ud h_1}{t_2}\\
    &= \lambda^2\int_{\IR_1}\int_{\II}\left|t_2\n_2(p_{t_2}(1-m_\lambda)(\h))\right|^2\,\cfrac{\ud h_2\ud t_2\ud h_1}{t_2}\\
    &\lesssim \lambda^2\left\|1-m_\lambda\right\|_{L^2(\IT)}^2\\&\lesssim\lambda^2\left|E_\lambda^c\right|.
\end{align*}
where in the equality above we used  the fact that $t_2\n_2(p_{t_2}(1))=0$, and in the second inequality we used 
$L^2$-boundedness of the Littlewood--Paley square function.

\smallskip

We now turn to $\frak{I}_2$.
Since 
\[\n_1\Phi(p_{t_1}p_{t_2}(m_{\lambda})(\h))=\Phi^\prime(p_{t_1}p_{t_2}(m_{\lambda})(\h))\n_1p_{t_1}p_{t_2}(m_{\lambda})(\h), \]
and recalling that $\widetilde{\psi}_{\bf R}(\h):=\psi_1(h_1/R_1)\psi_2(h_2/R_2)$,   we have, by an argument analogous to that in \eqref{dd1},
\begin{align}
     &\left|\n_2p^{\L_1}_{t_1}p_{t_2}^{\L_2}(f)(\h)\right|^2\left|\n_1\Phi(p_{t_1}p_{t_2}(m_{\lambda})(\h)\right|^2\widetilde{\psi}_{\bf R}(\h)\label{d4}\\
   &\leq \frac{1}{2}\n_2\n_2\left[\left|p^{\L_1}_{t_1}p_{t_2}^{\L_2}(f)(\h)\right|^2\left|\n_1\Phi(p_{t_1}p_{t_2}(m_{\lambda})(\h))\right|^2\right]\widetilde{\psi}_{\bf R}(\h)\notag\\
   &\quad-4p^{\L_1}_{t_1}p_{t_2}^{\L_2}(f)(\h)\n_2 p^{\L_1}_{t_1}p_{t_2}^{\L_2}(f)(\h)\n_1\Phi(p_{t_1}p_{t_2}(m_{\lambda})(\h))\n_2\n_1\Phi(p_{t_1}p_{t_2}(m_{\lambda})(\h))\widetilde{\psi}_{\bf R}(\h)\notag\\
   &\quad-\left|p^{\L_1}_{t_1}p_{t_2}^{\L_2}(f)(\h)\right|^2\left|\n_2\n_1\Phi(p_{t_1}p_{t_2}(m_{\lambda})(\h))\right|^2\widetilde{\psi}_{\bf R}(\h)\notag\\
   &\quad-\left|p^{\L_1}_{t_1}p_{t_2}^{\L_2}(f)(\h)\right|^2\n_1\Phi(p_{t_1}p_{t_2}(m_{\lambda})(\h))\n_2\n_2\n_1\Phi(p_{t_1}p_{t_2}(m_{\lambda})(\h))\widetilde{\psi}_{\bf R}(\h)\notag\\
   &=: \sum_{j=1}^4\mathscr{G}_j(h_1,h_2,t_1,t_2).\notag
\end{align}

One may apply the Cauchy--Schwarz inequality to obtain
\begin{align*}
    |{\mathscr G}_2(\h,{\bf t})|&\leq \frac{1}{2}\left|\n_2p^{\L_1}_{t_1}p_{t_2}^{\L_2}(f)(\h) \right|^2\left|\n_1\Phi(p_{t_1}p_{t_2}(m_{\lambda})(\h))\right|^2\psi_1(h_1/R_1)\psi_2(h_2/R_2)\\
    &\quad+32 \left|p^{\L_1}_{t_1}p_{t_2}^{\L_2}(f)(\h)\right|^2\left|\n_2\n_1\Phi(p_{t_1}p_{t_2}(m_{\lambda})(\h))\right|^2\psi_1(h_1/R_1)\psi_2(h_2/R_2)\\
    &=: G_1(h_1,h_2,t_1,t_2)+G_2(h_1,h_2,t_1,t_2).
\end{align*}
Note that $G_1$ can be absorbed  in the left-hand side of  \eqref{d4}.
Therefore $\frak{I}_2$ is controlled by 
\begin{align*}
   \bigg| &\ir \mathscr{G}_1(h_1,h_2,t_1,t_2)t_1t_2+G_2(h_1,h_2,t_1,t_2)t_1t_2\\
    &\quad+\mathscr{G}_3(h_1,h_2,t_1,t_2)t_1t_2+\mathscr{G}_4(h_1,h_2,t_1,t_2) t_1t_2\,{\ud \h\,\t}\bigg|\\
    &=: |\frak{I}_{21}+\frak{I}_{22}+\frak{I}_{23}+\frak{I}_{24}|.                     
\end{align*}

For $\frak{I}_{21}$, we have
\begin{align*}
    |\frak{I}_{21}|
    &\leq\frac{1}{2}\bigg|\ir\Delta_{h_2}\left[\left|p^{\L_1}_{t_1}p_{t_2}^{\L_2}(f)(\h)\right|^2\right.\\
    &\qquad\qquad\left.\left|\n_1\Phi(p_{t_1}p_{t_2}(m_{\lambda})(\h))\right|^2\right]t_1t_2\psi_1(h_1/R_1)\psi_2(h_2/R_2) \,\,{\ud \h\t}\bigg|\\
    &\quad+\frac{1}{2}\bigg|\ir \partial_{t_2}^2\big[\left|p^{\L_1}_{t_1}p_{t_2}^{\L_2}(f)(\h)\right|^2\\
    &\qquad\qquad\left.\left|\n_1\Phi(p_{t_1}p_{t_2}(m_{\lambda})(\h))\right|^2\right]t_1t_2\psi_1(h_1/R_1)\psi_2(h_2/R_2) \,\,{\ud \h\t}\bigg|\\
    &=: \frak{I}_{211}+\frak{I}_{212}.
\end{align*}

\noindent By an  argument similar to that for $\frak{I}_{12}$, we have $\frak{I}_{211}=0.$

\noindent By an  argument similar to that for $\widetilde{\frak{I}_{111}}$ and $\widetilde{\frak{I}_{12}}$, we obtain
\begin{align*}
    \frak{I}_{212}\leq C\lambda^2 \left\|1-m_\lambda\right\|_{L^2(\IT)}=C\lambda^2|E_\lambda^c|.
\end{align*}

For $\frak{I}_{22}$, by the Cauchy--Schwarz inequality, we have 
\begin{align*}
    |\frak{I}_{22}|&\leq C\int_{\III}\int_{\II} \left|p^{\L_1}_{t_1}p_{t_2}^{\L_2}(f)(\h)\right|^2\left|\n_2\n_1\Phi(p_{t_1}p_{t_2}(m_{\lambda})(\h))\right|^2t_1t_2\ud h_2\ud t_2\ud h_1\ud t_1\\
    &\leq C\int_{\III}\int_{\II} \left|p^{\L_1}_{t_1}p_{t_2}^{\L_2}(f)(\h)\right|^2\left|\Phi^{\prime\prime}(p_{t_1}p_{t_2}(m_{\lambda})(\h))\right|^2\\
    &\quad\quad\quad\times \left|\n_2p_{t_1}p_{t_2}(m_{\lambda})(\h)\right|^2\left|\n_1p_{t_1}p_{t_2}(m_{\lambda})(\h)\right|^2
    t_1t_2\ud h_2\ud t_2\ud h_1\ud t_1\\
    &\quad+ C\int_{\III}\int_{\II} \left|p^{\L_1}_{t_1}p_{t_2}^{\L_2}(f)(\h)\right|^2\left|\Phi^\prime(p_{t_1}p_{t_2}(m_{\lambda})(\h))\right|^2\\
    &\quad\quad\quad\times \left|\n_1\n_2p_{t_1}p_{t_2}(m_{\lambda})(\h)\right|^2t_1t_2\ud h_2\ud t_2\ud h_1\ud t_1\\
    &=: \frak{I}_{221}+\frak{I}_{222}.
\end{align*}
For $\frak{I}_{221}$, by the property of $\Phi^{\prime\prime}$ and an argument similar to that of $\widetilde{\frak{I}_{12}}$, we have that if  $\Phi^{\prime\prime}(p_{t_1}p_{t_2}(m_\lambda)(\h))\neq0$, then  $|p_{t_1}p_{t_2}(m_\lambda)(\h)|\leq \lambda.$ Therefore,
\begin{align*}
    \frak{I}_{221}&\leq C\lambda^2\int_{\III}\int_{\II} \left|\n_2p_{t_1}p_{t_2}(m_{\lambda})(\h)\right|^2\left|\n_1p_{t_1}p_{t_2}(m_{\lambda})(\h)\right|^2
    t_1t_2\,\,{\rm d\h \t}\\
    &\leq C\lambda^2 \int_{\III}\int_{\II} \left|t_2\n_2p_{t_1}p_{t_2}(m_{\lambda})(\h)\right|^2\left|t_1\n_1p_{t_1}p_{t_2}(m_{\lambda})(\h)\right|^2
    \,\,\cfrac{{\rm d\h \t}}{\bf t}.
\end{align*}
Let $\mathcal{M}_1$ and $\mathcal{M}_2$ be the Hardy--Littlewood maximal function on $\mathcal{G}_1$ and $\mathcal{G}_2$, respectively. Applying H\"older's inequality, we further have 
\begin{align*}
    \frak{I}_{221}
    &\leq C\lambda^2\left\{\int_{\IT}\left[\int_0^\infty\left|\mathcal{M}_1\left(|t_2\n_2p_{t_2}(m_{\lambda})|\right)(\h)\right|^2\,\,\cfrac{{\ud t_2}}{t_2}\right]^2\,\,{\ud h_1\ud h_2}\right\}^\frac{1}{2}\\
    &\quad\quad\times\left\{\int_{\IT}\left[\int_0^\infty\left|\mathcal{M}_2\left(|t_1\n_1p_{t_1}(m_\lambda)|\right)(\h)\right|^2\,\,\cfrac{{\ud t_1}}{t_1}\right]^2\,\,{\ud h_1\ud h_2}\right\}^\frac{1}{2}.
\end{align*}
By the $L^4$-boundedness of the vector-valued Hardy--Littlewood maximal operator and the Littlewood--Paley operator, together with the fact $\n_ip_{t_i}(1)=0$ for $i=1,2$, we have
\begin{align*}
    \frak{I}_{221}&\leq C\lambda^2\left\{\int_{\IT}\left[\int_0^\infty\left|t_2\n_2p_{t_2}(m_{\lambda})(\h)\right|^2\,\,\cfrac{{\ud t_2}}{t_2}\right]^2\,\,{\ud h_1\ud h_2}\right\}^\frac{1}{2}\\
    &\quad\quad\times\left\{\int_{\IT}\left[\int_0^\infty\left|t_1\n_1p_{t_1}(m_{\lambda})(\h)\right|^2\,\,\cfrac{{\ud t_1}}{t_1}\right]^2\,\,{\ud h_1\ud h_2}\right\}^\frac{1}{2}\\
    &\leq C\lambda^2\left\{\int_{\IT}\left[\int_0^\infty\left|t_2\n_2p_{t_2}(1-m_\lambda)(\h)\right|^2\,\,\cfrac{{\ud t_2}}{t_2}\right]^2\,\,{\ud h_1\ud h_2}\right\}^\frac{1}{2}\\
    &\quad\quad\times\left\{\int_{\IT}\left[\int_0^\infty\left|t_1\n_1p_{t_1}(1-m_\lambda)(\h)\right|^2\,\,\cfrac{{\ud t_1}}{t_1}\right]^2\,\,{\ud h_1\ud h_2}\right\}^\frac{1}{2}\\
    &\leq C\lambda^2\left|E_\lambda^c\right|.
\end{align*}
For $\frak{I}_{222}$,  note that for $(h_1,h_2)$ with $\Phi^\prime(p_{t_1}p_{t_2})(m_\lambda)(\h)\neq 0$, we have $p_{t_1}p_{t_2}(m_{\lambda})(\h)\leq \lambda$. As a consequence,  
\begin{align*}
    \frak{I}_{222}&\leq C\lambda^2 \int_{\III}\int_{\II} \left|t_1\n_1t_2\n_2p_{t_1}p_{t_2}(m_{\lambda})(\h)\right|^2 \,\,\cfrac{\ud h_2\ud t_2\ud h_1\ud t_1}{t_1t_2}\\
    &=C\lambda^2 \int_{\III}\int_{\II} \left|t_1\n_1t_2\n_2p_{t_1}p_{t_2}(1-m_\lambda)(\h)\right|^2 \,\,\cfrac{\ud h_2\ud t_2\ud h_1\ud t_1}{t_1t_2}\\
    &\leq C \lambda^2\left\|1-m_\lambda\right\|_{L^2(\IT)}^2\\
    &\leq C \lambda^2 |E_\lambda^c|.
\end{align*}
Hence, we have proved $|\frak{I}_{22}|\lesssim \lambda^2\left|E_\lambda^c\right|$.

By an argument similar to  that for $\frak{I}_{22}$, we obtain  
\begin{align*}
    \frak{I}_{23}\leq C\lambda^2\left|E_\lambda^c\right|.
\end{align*}

Finally, we turn to estimating $\frak{I}_{24}$. By the chain rule, we have
\begin{align*}
    &\n_1\Phi(p_{t_1}p_{t_2}(m_{\lambda})(\h))\n_2\n_2\n_1\Phi(p_{t_1}p_{t_2}(m_{\lambda})(\h))\\
    &=\Phi^\prime\Phi^{\prime\prime\prime}(p_{t_1}p_{t_2}(m_{\lambda})(\h))\left|\n_1p_{t_1}p_{t_2}(m_{\lambda})(\h)\right|^2\left|\n_2p_{t_1}p_{t_2}(m_{\lambda})(\h)\right|^2\\
    &\quad\quad+2\Phi^\prime\Phi^{\prime\prime}(p_{t_1}p_{t_2}(m_{\lambda})(\h))\n_1p_{t_1}p_{t_2}(m_{\lambda})(\h)\n_2p_{t_1}p_{t_2}(m_{\lambda})(\h)\n_1\n_2p_{t_1}p_{t_2}(m_{\lambda})(\h).
\end{align*}

Thus, $\frak{I}_{24}$ can be dominated by $\frak{I}_{241}+\frak{I}_{242}$ with respect to the above two terms in the integrand respectively. Similar to $\frak{I}_{221}$, we have  
\begin{align}\label{e-24-1}
    |\frak{I}_{241}|\lesssim \lambda^2|E_\lambda^c|.
\end{align}
By the support property of $\Phi^\prime\Phi^{\prime\prime}$ and H\"older's inequality, we obtain
\begin{align*}
    \frak{I}_{242}&\leq C\lambda^2\int_{\II}\int_{\III} \left|\n_1p_{t_1}p_{t_2}(m_{\lambda})(\h)\n_2p_{t_1}p_{t_2}(m_{\lambda})(\h)\right|\\
    &\quad\quad\times\left|\n_1\n_2p_{t_1}p_{t_2}(m_{\lambda})(\h)\right|t_1t_2{\ud h_2\ud t_2\ud h_1\ud t_1}\\
    &\leq C\lambda^2\left\{\int_{\II}\int_{\III} \left|\n_1p_{t_1}p_{t_2}(m_{\lambda})(\h)\n_2p_{t_1}p_{t_2}(m_{\lambda})(\h)\right|^2t_1t_2{\ud t_1\ud h_1\ud t_2\ud h_2}\right\}^{1/2}\\
    &\quad\quad\times\left\{\int_{\II}\int_{\III}\left|\n_1\n_2p_{t_1}p_{t_2}(m_{\lambda})(\h)\right|^2t_1t_2{\ud h_2\ud t_2\ud h_1\ud t_1}\right\}^{1/2}.
\end{align*}
Similar to $\frak{I}_{221}$ and $\frak{I}_{222}$, we have $|\frak{I}_{242}|\lesssim \lambda^2|E_\lambda^c|
$. This, together with \eqref{e-24-1}, yields 
$$
|\frak{I}_{24}|\lesssim \lambda^2|E_\lambda^c|.
$$

Combining these estimates together, we conclude 
\begin{align*}
    \int_{A_\lambda}\left|\Snp(f)(\g)\right|^2\,\,{\rm d\g}\lesssim \int_{E_\lambda}\left|N^\beta_{\L_1,\L_2}(f)(\h)\right|^2 \rm{d}\h +\lambda^2|E_\lambda^c|.
\end{align*}

The proof is complete.
\end{proof}
    
\begin{corollary}\label{lemma auxiliary}
The area integral function $\SP f$ defined in (\ref{esf}) satisfies
\begin{align*}
\big| \big\{ \g \in \IT:\ |\SP f(\g)|>\lambda \big\}\big|\lesssim \bigg\|{f \over
\lambda}\bigg\|_{ L \log^+ L(\IT)}.
\end{align*}

\end{corollary}

\begin{proof}
We adopt the notation used in Theorem \ref{prod F-S thm}. Actually, from Theorem \ref{prod F-S thm}, we obtain
\begin{align}\label{eq:good-lambda}
    &\left|\big\{{\bf g}\in\IT\colon \Snp(f)(\g)>\lambda\big\}\right| \lesssim\left|E_\lambda^c\right|+\frac{1}{\lambda^{2}}\int_{E_\lambda}|N^\beta_{\L_1,\L_2}(f)(\g)|^2\,\mathrm{d}\g.
\end{align}

It follows from  the endpoint estimate of the maximal operator (see \cite[Proposition 4.1]{CLLP2025}) that $\left|E_\lambda^{c}\right|\lesssim \|f/\lambda\|_{L\log^+L(\IT)}$. This, together with the layer-cake representation, yields,
\begin{align*}
&\left|\big\{{\bf g}\in\IT\colon \Snp(f)(\g)>\lambda\big\}\right| \\
&\lesssim\left|E_\lambda^{c}\right|+\frac{1}{\lambda^{2}}\int_{0}^{\lambda}2\mu\left|E_\mu^{c}\right|\,\mathrm{d\mu}\\
&\lesssim \|f/\lambda\|_{L\log^+L(\IT)}+\frac{1}{\lambda^{2}}\int_{0}^{\lambda}\|f/\mu\|_{L\log^+L(\IT)}\mu\,\mathrm{d}\mu\\
&\lesssim \|f/\lambda\|_{L\log^+L(\IT)}+\frac{1}{\lambda}\int_{0}^{\lambda}\log\left(e+\frac{\lambda}{\mu}\right)\left\|\frac{f}{\lambda}\right\|_{L\log^+L(\IT)}\,\mathrm{d}\mu\\
&\lesssim \|f/\lambda\|_{L\log^+L(\IT)}.
\end{align*}
 This, together with the fact
$\SP(f)(\g)\leq \Snp(f)(\g)$, implies that 
 for all $\lambda\in\mathbb{R}_{+}$,
$$
\big|\big\{{\g}\in\IT\colon|\SP(f)(\g)|>\lambda\big\}\big| 
\lesssim \|f/\lambda\|_{L\log^+L(\IT)}.
$$
The proof is complete.
 \end{proof}

\medskip

 \section{Proof of Theorem \ref{lemma L log L atom}}\label{atomic decomposition}
 \setcounter{equation}{0}

In this section, we develop an atomic decomposition for the space $L\log^+ L(\IT)$, which plays a key role in the proof of Theorem \ref{thm1}. Inspired by \cite{FSt1982,HLMMY}, we provide the proof as follows.














\begin{proof}
Assume that $f\in L\log^+L(\IT)\cap L^2(\IT)$. Now for any $k\in\mathbb{Z}$, we set
\begin{eqnarray*}
\Omega_k&=&\{ \g\in\IT: \SP (f)(\g)>2^k \};\\[.2cm]
\widetilde{\Omega}_k&=& \{\g\in\IT: \mathcal{M}_s(\chi_{\Omega_k})(\g)>C_2/2 \};\\[.2cm]
\Omega_k^\dagger&=&\bigcup_{S\in \mathscr{M}(\widetilde{\Omega}_k)}3S;\\[.2pt]
\frak{B}_k&=&\{R\in\mathscr{R}(\IR_1\times\IR_2):\ |R\cap\Omega_k|\geq {1/
2}|R|,\ |R\cap\Omega_{k+1}|< {1/ 2}|R|\};
\end{eqnarray*}
where $C_2$ is as in Lemma \ref{R cap Omega>1/2R}.
For  $R=I\times J$,  we define the tent $R_+$ over $R$ to be the set 
\begin{align}
 R_+:=\big\{(\h,{\bf t}):\h \in R, {4C_*\l(I)}<t_1\leq 8C_*\l(I), {4C_*\l(J)}<t_2\leq 8C_*\l(J)\big\},
 \end{align}
where the constant $C_*$ is as in Theorem \ref{thm:hytonen-kairema}. This definition implies that if $(\h,\bf t)\in$ $R_+$, then 
\begin{align}
    R\subset B_1(h_1,4C_*\l(I))\times B_2(h_2,4C_*\l(J))\subset B_1(h_1,t_1)\times B_2(h_2,t_2),
\end{align}
and so $R_+\subset \Gamma(\g)$ for all $\g\in R$. Further, the space $(\IR_1\times\mathbb{R}_+)\times(\IR_2\times\mathbb{R}_+)$ is the disjoint union of all tents $R_+$ as $R$ runs over $\mathscr{R}(\IT)$.

 By  definition,  $\Omega_{k+1}\subset\Omega_k$, so the sequence $|R\cap \Omega_k|/|R|$ is decreasing and the sets $\frak{B}_k$ are pairwise disjoint. If $g\in R\in \frak{B}_k$, then $|R\cap \Omega_k|/|R|\geq 1/2$, so $g$ $\in \widetilde{\Omega}_k$ by Lemma \ref{R cap Omega>1/2R}. 
By the definition of $\frak{B}_k$, 
\begin{align}\label{5.2}
    |R\cap (\widetilde{\Omega}_k\backslash\Omega_{k+1})|=|R\backslash(R\cap \Omega_{k+1})|\geq \frac{1}{2}|R|.
\end{align}

 Let $\Phi$ be as in Lemma \ref{lemma finite speed}, define
$\psi(s)=c_{\Phi} s^{2M}\Phi(s)$, and recall that
 $Q_{t_1t_2}f=t_1\sqrt{\L_1}e^{-t_1\sqrt{\L_1}}\otimes t_2\sqrt{\L_2}e^{-t_2\sqrt{\L_2}}f$. From the
spectral theory \cite{Yo}, we have
\begin{eqnarray}\label{5.3}
f(\g)&=& \int_0^\infty\int_0^\infty\psi(t_1\sqrt{\L_1})\psi(t_2\sqrt{\L_2})Q_{t_1t_2}f(\g){\ud t_1\ud t_2 \over
t_1 t_2}\\
&=&\int_{(\IR_1\times\mathbb{R}_+)\times(\IR_2\times\mathbb{R}_+)}K_{\psi(t_1\sqrt{\L_1})}(g_1,h_1)K_{\psi(t_2\sqrt{\L_2})}(g_2,h_2)Q_{t_1t_2}f(\h)
{{{\bf {\rm d\h}}} \t \over t_1 t_2}\notag\\
&=&\sum_k\sum_{R\in
\frak{B}_k}\int_{R_+}K_{\psi(t_1\sqrt{\L_1})}(g_1,h_1)K_{\psi(t_2\sqrt{\L_2})}(g_2,h_2)Q_{t_1t_2}f(\h){{{\bf {\rm d\h}}} \t \over t_1 t_2}\notag\\
&=:&\sum_k a_k(\g),\notag
\end{eqnarray}

\noindent where $K_{\psi(t_i\sqrt{\L_i})}(g_1,h_1)$ denotes the kernel of $\psi(t_i\sqrt{\L_i}))$, $i=1,2.$

Next, we proceed by associating to each dyadic $R\in \frak{B}_k$ a maximal
dyadic subrectangle $\tilde{R}\in \mathscr{M}(\widetilde{\Omega}_k)$ such that
$R\subset \tilde{R}$. Then for each $S\in \mathscr{M}(\widetilde{\Omega}_k),$
set
 $$ a_{k,S}=\sum_{R\in \frak{B}_k:\ \tilde{R}=S }\hat{a}_{k,R},
 $$
 where
$$\hat{a}_{k,R}(\g):=\int_{R_+}K_{\psi(t_1\sqrt{\L_1})}(g_1,h_1)K_{\psi(t_2\sqrt{\L_2})}(g_2,h_2)Q_{t_1t_2}f(\h){{{\bf {\rm d\h}}} \t \over t_1 t_2}.$$
Then we can write
$$a_k=\sum\limits_{S\in
\mathscr{M}(\widetilde{\Omega}_k)} a_{k,S}.
$$
Observe that 
$$
\psi(t_i\sqrt{\L_i}) = c_{\Phi} t_i^{2M}\L_i^M\Phi(t_i\sqrt{\L_i}), \quad i=1,2.
$$
Arguing as in \eqref{5.3}, we  express $f(\g)$ as
\begin{align*}
f(\g)&=\int_0^\infty\int_0^\infty \big(t_1^{2M}\L_1^M\Phi(t_1\sqrt{\L_1})\big) \big(t_2^{2M}\L_2^M\Phi(t_2\sqrt{\L_2})\big) Q_{t_1t_2}f(\g) \frac{\ud t_1\ud t_2}{t_1 t_2}\\
&=\sum_k\sum_{R\in \frak{B}_k}\L_1^M\otimes\L_2^M \int_{R_+} t_1^{2M} K_{\Phi(t_1\sqrt{\L_1})}(g_1,h_1) t_2^{2M}K_{\Phi(t_2\sqrt{\L_2})}(g_2,h_2) Q_{t_1t_2}f(\h) \frac{\ud \h \ud t_1 \ud t_2}{t_1 t_2}.
\end{align*}
From this factorization, we can write the previously obtained $a_k(\g)$ and $a_{k,S}(\g)$ as 
$$a_k(\g)=\L_1^M\otimes \L_2^M b_k(\g) \quad \text{and} \quad a_{k,S}(\g)=\L_1^M\otimes\L_2^M b_{k,S}(\g).$$ 
Here, we define $b_{k} = \sum_{S\in \mathscr{M}(\widetilde{\Omega}_k)} b_{k,S}$, with each $b_{k,S}$ given by
\begin{equation*}
b_{k,S}(\g) = \sum_{R\in \frak{B}_k:\widetilde{R}=S} \int_{R_+} t_1^{2M} K_{\Phi(t_1\sqrt{\L_1})}(g_1,h_1) t_2^{2M} K_{\Phi(t_2\sqrt{\L_2})}(g_2,h_2) Q_{t_1t_2}f(\h) \frac{\ud h_1 \ud t_1}{t_1} \frac{\ud h_2 \ud t_2}{t_2}.
\end{equation*}
\smallskip
For each $k$ and $S\in \mathscr{M}(\widetilde{\Omega}_k)$, it follows by Lemma \ref{lemma finite speed} that
$\L_1^{i_1}\otimes \L_2^{i_2}(b_{k,S})$ is supported in $3S$ for all
$0\leq i_1,i_2\leq M$, and hence $a_k$ is supported in
$\Omega^\dagger_k$. 
Invoking \cite[Lemma 3.5]{CLLP2025} alongside the strong maximal function theorem and Corollary~\ref{lemma auxiliary}, we  obtain that
$$ |\Omega_k^\dagger|\lesssim|\widetilde{\Omega}_k|\leq C|\Omega_k|\leq C\|2^{-k}f\|_{L\log^+L(\IT)}. $$

Next, we claim that for each $k\in{\mathbb Z}$,
\begin{eqnarray}\label{claim LlogL atom a k}
 \|a_k\|_{L^2(\IT)}^2\leq C2^{2k}\|2^{-k}f\|_{L\log^+L(\IT)},
\end{eqnarray}

\noindent
and for  every $0\leq i_1,i_2\leq M$,

\begin{eqnarray}
 \sum_{S\in
\mathscr{M}(\widetilde{\Omega}_k)}\l(I)^{-4M}\l(J)^{-4M}\big\|(\l(I)^2\L_1)^{i_1}\otimes(\l(J)^2\L_2)^{i_2}
b_{k,S}\big\|_{L^2(\IT)}^2\lesssim 2^{2k}\left\|{f\over 2^{k}}\right\|_{L\log^+L(\IT)}. \label{claim LlogL atom b k}
\end{eqnarray}

\noindent

\smallskip

To see this, for any $\varphi\in
L^2(\IT)$ with $\|\varphi\|_{L^2(\IT)}=1$, we
have

\begin{eqnarray*}\label{eah1}
 \big|\langle a_k,\varphi\rangle\big|
&=&\bigg|\sum_{R\in \frak{B}_k}
\int_{R_+}\psi(t_1\sqrt{\L_1})\psi(t_2\sqrt{\L_2})\varphi(\h)
Q_{t_1t_2}f(\h){{{\bf {\rm d\h}}} \t \over t_1 t_2}\bigg|\\[5pt]
&\leq& \bigg( \sum_{R\in \frak{B}_k} \int_{R_+} \big|
\psi(t_1\sqrt{\L_1})\psi(t_2\sqrt{\L_2})\varphi(\h)
\big|^2{{{\bf {\rm d\h}}} \t \over t_1 t_2} \bigg)^{1/2}\\
&&\hskip 1cm \times\bigg( \sum_{R\in \frak{B}_k} \int_{R_+} \big|
Q_{t_1t_2}f(\h)\big|^2{{{\bf {\rm d\h}}} \t \over t_1 t_2} \bigg)^{1/2}\\[5pt]
&\leq& C\|\varphi\|_{L^2(\IT)} \bigg(
\sum_{R\in \frak{B}_k} \int_{R_+} \big| Q_{t_1t_2}f(\h)\big|^2{{{\bf {\rm d\h}}}
\t \over t_1 t_2} \bigg)^{1/2}.
\end{eqnarray*}

\noindent To continue, we observe  that

\begin{eqnarray*}\label{eah2}
\int_{\widetilde{\Omega}_k\setminus \Omega_{k+1}}
|\SP (f)(\g)|^2{\rm d\g}  &\leq&
C2^{2k}|\widetilde{\Omega}_k|\leq C2^k \cdot 2^k|\Omega_k|\leq C2^{2k}\|2^{-k}f\|_{L\log^+L(\IT)}.
\end{eqnarray*}

\noindent Combining this with \eqref{5.2}, we deduce that

\begin{eqnarray*}\label{eah3}
&&\int_{\widetilde{\Omega}_k\setminus \Omega_{k+1}}
|\SP (f)(\g)|^2{\rm d\g}  \\[5pt]
&&=\int\limits_{(\IR_1\times \mathbb{R}_+)\times(\IR_2\times \mathbb{R}_+)}\Big|
\big\{\g\in\widetilde{\Omega}_k\setminus \Omega_{k+1}:\
\rho_i(g_i,h_i)<t_i, i=1,2 \big\} \Big|
  \big|
Q_{t_1t_2}f(\h)\big|^2{{{\bf {\rm d\h}}} \t \over t_1^{Q_1+1} t_2^{Q_2+1}}\\[5pt]
&&\geq \sum_{R\in \frak{B}_k}\int_{R_+}\Big|
\big\{\g\in\widetilde{\Omega}_k\setminus \Omega_{k+1}:\
\rho_i(g_i,h_i)<t_i, i=1,2 \big\} \Big|
  \big|
Q_{t_1t_2}f(\h)\big|^2{{{\bf {\rm d\h}}} \t \over t_1^{Q_1+1} t_2^{Q_2+1}}\\[5pt]
&&\geq {1\over 2}\sum_{R\in \frak{B}_k} \int_{R_+} \big|
Q_{t_1t_2}f(\h)\big|^2{{{\bf {\rm d\h}}} \t \over t_1 t_2}.
\end{eqnarray*}

\noindent
From  the estimates above, we have 
$$|\langle a_k,\varphi\rangle|\leq C2^{k}\|2^{-k}f\|^{1/2}_{L\log^+L(\IT)}.$$
This implies (\ref{claim LlogL atom a k}).
 The proof of (\ref{claim LlogL atom b k}) is similar to that of (\ref{claim LlogL atom a k}), and so we omit it.
This completes the proof of Theorem \ref{lemma L log L atom}.
\end{proof}

\medskip

\section{Proof of Theorem \ref{thm1}
}\label{application}
\setcounter{equation}{0}

In this section, we obtain endpoint estimates for a class of singular integrals with non‑smooth kernels, including the area operator and the double Riesz transforms associated with Schrödinger operators.

Suppose that $f\in L^2(\IT)\bigcap L(\log^+L)(\IT)$. By Theorem \ref{lemma L log L atom},   one may write
\begin{align*}
f=\sum_{k\in\mathbb{Z}}a_k,\,\,a_k=\sum_{S\in\mathscr{M}(\widetilde{\Omega}_k)}a_{k,S},\,\,\text{and}\,\, a_{k,S}=(\L_1^M\otimes \L^M_2)b_{k,S},
\end{align*}
where $a_k$ and $a_{k,S}$ satisfy (1) and (2) of  Theorem \ref{lemma L log L atom}.

Now let us define
\begin{align*}
    \widetilde{\widetilde{\Omega}}_k:=\{\g\in\IT:\,\mathcal{M}_s(\chi_{\widetilde{\Omega}_k})(\g)>C_2/2\},
\end{align*}
and
\begin{align*}
    \widetilde{\widetilde{\widetilde{\Omega}}}_k:=\{\g\in\IT:\,\mathcal{M}_s(\chi_{\widetilde{\widetilde{\Omega}}_k})(\g)>C_2/2\}.
\end{align*}

\subsection{Proof of \eqref{estimate1.2} for the area function $\S$}\label{4.1}


 Noting that $\S$ is a non-negative sublinear operator  and $L^2(\IT)\bigcap L(\log^+ L)(\IT)$ is dense in $L(\log^+ L)(\IT)$ (\cite[Proposition 2.6]{CLLP2025}), it suffices to prove that 
 \begin{align}\label{T from L log L to weak L1}
 \bigg|\{\g\in\IT: \S(f)(\g)>1 \}\bigg|\lesssim \|f\|_{L\log^+L(\IT)},  
 \end{align}
for all $f\in L^2(\IT)\bigcap L(\log^+ L)(\IT)$.

For $f\in L\log^+L(\IT)\bigcap L^2(\IT)$.   Theorem \ref{lemma L log L atom} gives $f=\sum_ka_k$ and 
\begin{align}\label{S-k<0}
    \Big\|\sum_{k\leq 0}a_k\Big\|_{L^2(\IT)}&\leq \sum_{k\leq0}\|a_k\|_{L^2(\IT)}\nonumber\\
    &\lesssim \sum_{k\leq0}2^k\|2^{-k}f\|_{L\log^+L(\IT)}^{1/2}\lesssim \|f\|_{L\log^+L(\IT)}^{1/2} .
\end{align}
Then we apply  the $L^2(\IT)$ boundedness of $\S$ to obtain
\begin{align}
    \left|\left\{\g\in\IT:\,\S\left(\sum_{k\leq 0}a_k\right)(\g)>1\right\}\right|\lesssim \Big\|\sum_{k\leq0}a_k\Big\|_{L^2(\IT)}^2\lesssim \|f\|_{L\log^+L(\IT)}.\label{3.17}
\end{align}

To estimate $\S(\sum_{k\geq 1}a_k)$,  consider the atoms $a_k,\,k>0$ and their structure. Each $a_k$ is supported in $\Omega_k^\dagger$ with $|\Omega_k^\dagger|\lesssim \|2^{-k}f\|_{L\log^+L(\IT)}$, and $a_k=\sum_{S\in \mathscr{M}(\widetilde{\Omega}_k)  }a_{k,S}$.

For all $S=I\times J\subset\widetilde{\Omega}_k$, let $\hat{I}$ be the largest dyadic cube containing $I$ such that $\hat{I}\times J\subset \widetilde{\widetilde{\Omega}}_k$. Next, let $\hat{J}$ be the largest dyadic cube containing $J$ such that $\hat{I}\times\hat{J}\subset \widetilde{\widetilde{\widetilde{\Omega}}}_k$. Finally, let $S^{\dagger}$ be $100\beta_k(\hat{I}\times\hat{J})$, where $\beta_k=2^{k/(2Q_1+2Q_2)}$. 
Applying \cite[Lemma 3.5]{CLLP2025}, the strong maximal function theorem and Corollary~\ref{lemma auxiliary},
\begin{align}\label{3.11}
    \bigg|\bigcup_{S\in \mathscr{M}(\widetilde{\Omega}_k)}S^\dagger\bigg|\lesssim2^{k/2}|\widetilde{\widetilde{\widetilde{\Omega}}}_k|\lesssim 2^{k/2}|\Omega_k|\lesssim 2^{k/2}\|2^{-k}f\|_{L\log^+L(\IT)}.
\end{align}

We claim that there exists some $\delta>0$ such that
\begin{align}\label{claim s.1}
    \sum_{S\in \mathscr{M}(\widetilde{\Omega}_k)}\int_{(S^\dagger)^c}|\S(a_{k,S})(\g)|\,{\rm d\g}\lesssim 2^{-\delta k}\|f\|_{L\log^+L(\IT)},\quad \forall k\in\mathbb{N}.
\end{align}

Assume \eqref{claim s.1} holds and let $E^\dagger:=\bigcup_{k\geq1}\bigcup_{S\in\mathscr{M}(\widetilde{\Omega}_k)}S^\dagger$.
It follows that
\begin{align*}
\int_{(E^\dagger)^c}\Big|\S\Big(\sum_{k\geq1} a_{k}\Big)(\g)\Big|\,\ud\g&\lesssim \sum_{k\geq 1}\sum_{S\in\mathscr{M}(\widetilde{\Omega}_k)}
\int_{(S^\dagger)^c}|\S(a_{k,S})(\g)|\,{\ud \g}\\
&\lesssim \sum_{k\geq 1}2^{-\delta k}\|f\|_{L\log^+L(\IT)}\\
&\lesssim \|f\|_{L\log^+L(\IT)}.
\end{align*}
On the other hand, from  \eqref{3.11} we obtain
\begin{align*}
    |E^\dagger|\lesssim \sum_{k\geq1}2^{k/2}\|2^{-k}f\|_{L\log^+L(\IT)}\lesssim \|f\|_{L\log^+L(\IT)} .
\end{align*}
By Chebyshev's inequality, we have
\begin{align*}
    &\bigg|\bigg\{\g\in\IT:\,\S\left(\sum_{k\geq1}a_k\right)(\g)>1\bigg\}\bigg|\\
    &\qquad\leq |E^\dagger|+\int_{(E^\dagger)^c}\bigg|\S\left(\sum_{k\geq1}a_k\right)(\g)\bigg|\,{\rm d}\g\\
    &\qquad\lesssim \|f\|_{L\log^+L(\IT)}.
\end{align*}
This, together with \eqref{3.17},  implies \eqref{T from L log L to weak L1}.

\smallskip

It remains to prove \eqref{claim s.1}. Now 
\begin{align*}
    \int_{(S^\dagger)^c}&\big|\S(a_{k,S})(\g)\big|\,\rm{d}\g\\
    &\leq \int_{(100\beta_k\hat{I})^c}\int_{100J}\big|\S(a_{k,S})(\g)\big|\,{\rm d\g}   +\int_{(100\beta_k\hat{I})^c}\int_{(100J)^c}\big|\S(a_{k,S})(\g)\big|\,\rm{d}\g\\    &\qquad+\int_{(100\beta_k\hat{J})^c}\int_{100I}\big|\S(a_{k,S})(\g)\big|\,{\rm d\g} +\int_{(100\beta_k\hat{J})^c}\int_{(100I)^c}\big|\S(a_{k,S})(\g)\big|\,\rm{d}\g\\
    &=:\frak{I}_1+\frak{I}_2+\frak{I}_3+\frak{I}_4.
\end{align*}
It suffices to control $\frak{I}_1$ and $\frak{I}_2$, as the other two terms are similar.

By H\"older's inequality and the $L^2(\mathcal{G}_2)$ boundedness of $\S$,
\begin{align}
    \frak{I}_1&\lesssim |J|^{1/2} \int_{(100\beta_k\hat{I})^c}\left(\int_{100J}\left|\S^{[1]}(a_{k,S})(g_1,g_2)\right|^2\,{\ud g_2}\right)^{1/2}\,\ud g_1,\label{I_1}
\end{align}
where 
\begin{align*}
    |\S^{[1]}(a_{k,S})(g_1,g_2)|^2&:=\iint_{\Gamma(g_1)}|t_1^2\L_1e^{-t_1^2\L_1}(a_{k,S})(h_1,g_2)|^2\,\cfrac{\ud h_1\ud t_1}{t_1^{Q_1+1}}.
\end{align*}
To estimate it, we split the integral as follows:
 \begin{align*}
     |\S^{[1]}(a_{k,S})(g_1,g_2)|^2     &=\int_0^{\l(I)}\int_{\rho_1(h_1,g_1)<t_1}|t_1^2\L_1e^{-t_1^2\L_1}(a_{k,S})(h_1,g_2)|^2\,\cfrac{\ud h_1\ud t_1}{t_1^{Q_1+1}}\\
     &\qquad+\int_{\l(I)}^\infty\int_{\rho_1(h_1,g_1)<t_1}|t_1^2\L_1e^{-t_1^2\L_1}(a_{k,S})(h_1,g_2)|^2\,\cfrac{\ud h_1\ud t_1}{t_1^{Q_1+1}}\\
     &=:\frak{I}_{11}+\frak{I}_{12}.
 \end{align*}
 
For $\frak{I}_{11}$, by the heat kernel estimate, we have
\begin{align*}
    |t_1^2\L_1e^{-t_1^2\L_1}(a_{k,S})(h_1,g_2)|&\lesssim \int_{\G}\cfrac{t_1}{(t_1+\rho_1(h_1,z_1))^{Q_1+1}}|a_{k,S}(z_1,g_2)|\,\ud z_1.
\end{align*}
By the support restriction on $a_{k,S}$ in the first variable, we have $z_1\in 3I$. If $g_1\notin100\beta_k\hat{I}$ and $\rho_1(h_1,g_1)<t_1$,   then $\rho_1(g_1,g_I)\approx\rho_1(g_1,z_1)\lesssim\rho_1(g_1,h_1)+\rho_1(h_1,z_1)<t_1+\rho_1(h_1,z_1)$, where $g_I$ denotes the center of cube $I$. Furthermore, one may apply H\"older's  inequality to obtain
\begin{align*}
    |t_1^2\L_1e^{-t_1^2\L_1}(a_{k,S})(h_1,g_2)|&\lesssim \cfrac{t_1}{\rho_1(g_1,g_I)^{Q_1+1}}\|a_{k,S}(\cdot,g_2)\|_{L^1(\G)}\\
    &\lesssim |I|^{\frac{1}{2}}\cfrac{t_1}{\rho_1(g_1,g_I)^{Q_1+1}}\|a_{k,S}(\cdot,g_2)\|_{L^2(\G)}.
\end{align*}
We deduce that  
\begin{align}\label{ee-S-I11}
    \frak{I}_{11}&\lesssim \int_0^{\l(I)}\int_{\rho_1(h_1,g_1)<t_1} |I|\cfrac{t_1^2}{\rho_1(g_1,g_I)^{2(Q_1+1)}}\|a_{k,S}(\cdot,g_2)\|^2_{L^2(\G)}\ \,\cfrac{\ud h_1\ud t_1}{t_1^{Q_1+1}}\nonumber\\
    &\approx |I|\int_0^{\l(I)}\cfrac{t_1^2}{\rho_1(g_1,g_I)^{2(Q_1+1)}}\|a_{k,S}(\cdot,g_2)\|^2_{L^2(\G)}\ \,\cfrac{\ud t_1}{t_1}\\
&\approx\cfrac{\l(I)^2|I|}{\rho_1(g_1,g_I)^{2(Q_1+1)}}\|a_{k,S}(\cdot,g_2)\|_{L^2(\G)}^2. \nonumber
\end{align}

Let us estimate  $\frak{I}_{12}$. To  simplify the notation, we introduce
$$B_{k,S}(\cdot,g_2):=\l(I)^{-2}\l(J)^{-2}(\l(I)^2\L_1)^0\otimes(\l(J)^2\L_2)(b_{k,S})(\cdot,g_2).$$
By extracting the scaling factor $\l(I)/t_1$, one can express the term as
\begin{align*}
    |t_1^2\L_1e^{-t_1^2\L_1}a_{k,S}(h_1,g_2)|
    &= \left|t_1^2\L_1e^{-t_1^2\L_1}(\L_1\otimes\L_2)b_{k,S}(h_1,g_2)\right|\\
    &= \left(\frac{\l(I)}{t_1}\right)^2 \left| (t_1^2\L_1)^2e^{-t_1^2\L_1} B_{k,S}(h_1,g_2) \right|.
\end{align*}
Applying the heat kernel estimate \eqref{estimate heat kernel} with $k=2$, we have
\begin{align*}
    |t_1^2\L_1e^{-t_1^2\L_1}a_{k,S}(h_1,g_2)|
    &\lesssim \left(\frac{\l(I)}{t_1}\right)^2 \int_{\G} t_1^{-Q_1}e^{-\frac{\rho_1(h_1,z_1)^2}{t_1^2}} |B_{k,S}(z_1,g_2)| \ud z_1 \\
    &\lesssim \left(\frac{\l(I)}{t_1}\right)^2 \int_{\G} \frac{t_1}{(t_1+\rho_1(h_1,z_1))^{Q_1+1}} |B_{k,S}(z_1,g_2)| \ud z_1.
\end{align*}
Since $z_1\in 3I$, $g_1\notin100\beta_k\hat{I}$ and $\rho_1(g_1,h_1)<t_1$,  we have 
\begin{align*}
    \rho_1(g_1,g_I)\approx\rho_1(g_1,z_1)\lesssim\rho_1(h_1,z_1) +t_1.
\end{align*}
Inserting this into the previous integral yields
\begin{align*}
    \left|t_1^2\L_1e^{-t_1^2\L_1}a_{k,S}(h_1,g_2)\right|
    \lesssim \left(\cfrac{\l(I)}{t_1}\right)^2\cfrac{t_1}{\rho_1(g_1,g_I)^{Q_1+1}}\|B_{k,S}(\cdot,g_2)\|_{L^1(\G)}.
    \end{align*}
Substituting this bound into the integral defining $\frak{I}_{12}$, we deduce that
\begin{align}\label{ee-S-I12}
   \frak{I}_{12}&\lesssim\int_{\l(I)}^\infty \int_{\rho_1(h_1,g_1)<t_1} \left(\frac{\l(I)}{t_1}\right)^4\cfrac{t_1^2}{\rho_1(g_1,g_I)^{2(Q_1+1)}}\|B_{k,S}(\cdot,g_2)\|_{L^1(\G)}^2\cfrac{\ud h_1\ud t_1}{t_1^{Q_1+1}}\nonumber\\
&\approx\int_{\l(I)}^\infty \left(\cfrac{\l(I)}{t_1}\right)^4\cfrac{t_1^2}{\rho_1(g_1,g_I)^{2(Q_1+1)}}\|B_{k,S}(\cdot,g_2)\|_{L^1(\G)}^2\cfrac{\ud t_1}{t_1}\\
    &\approx \cfrac{\l(I)^2\|B_{k,S}(\cdot,g_2)\|_{L^1(\G)}^2}{\rho_1(g_1,g_I)^{2(Q_1+1)}}.\nonumber
\end{align}

From above estimates, we have
\begin{align*}
    \frak{I}_1&\lesssim |J|^{1/2}\int_{(100\beta_k\hat{I})^c}\left(\int_{100J}\frak{I}_{11}+\frak{I}_{12}\,{\ud  g_2}\right)^{1/2}\,\ud g_1\\
    &\lesssim |J|^{1/2}\int_{(100\beta_k\hat{I})^c}\left(\int_{100J}\frak{I}_{11}\,{\ud g_2}\right)^{1/2}\,\ud g_1+|J|^{1/2}\int_{(100\beta_k\hat{I})^c}\left(\int_{100J}\frak{I}_{12}\,{\ud g_2}\right)^{1/2}\,\ud g_1
\end{align*}

By \eqref{ee-S-I11}, we have
\begin{align}\label{ee-S-I110}
&|J|^{1/2}\int_{(100\beta_k\hat{I})^c}\left(\int_{100J}\frak{I}_{11}\,{\ud g_2}\right)^{1/2}\,\ud g_1\nonumber\\
    &\lesssim |S|^{\frac{1}{2}}\int_{(100\beta_k\hat{I})^c}\left(\int_{100J}\|a_{k,S}(\cdot,g_2)\|^2_{L^2(\G)}\,\cfrac{\l(I)^2}{\rho_1(g_1,g_I)^{2Q_1+2}}dg_2\right)^{\frac{1}{2}}\ud g_1\\
    &\lesssim |S|^{\frac{1}{2}}(\beta_k)^{-1}\gamma_1(S)^{-1}\|a_{k,S}\|_{L^2(\IT)},\nonumber
\end{align}
where in the last inequality we have used the definition of $\gamma_1(S)$.

 Notice that the support of $B_{k,S}(\cdot,g_2)$ in the first variable is contained in $3I$. By H\"older's inequality, we have $$\|B_{k,S}(\cdot,g_2)\|_{L^1(\G)} \lesssim |I|^{\frac{1}{2}}\|B_{k,S}(\cdot,g_2)\|_{L^2(\G)}.$$
This estimate, together with  \eqref{ee-S-I12}, implies 
\begin{align}\label{ee-S-I120}
& |J|^{1/2}\int_{(100\beta_k\hat{I})^c}\left(\int_{100J}\frak{I}_{12}\,{\ud g_2}\right)^{1/2}\,\ud g_1 \nonumber\\  
     &\lesssim |I|^{\frac{1}{2}}|J|^{\frac{1}{2}} \int_{(100\beta_k\hat{I})^c} \frac{\ell(I)}{\rho_1(g_1,g_I)^{Q_1+1}} \ud g_1 \left(\int_{100J}\|B_{k,S}(\cdot,g_2)\|_{L^2(\G)}^2 \ud g_2\right)^{\frac{1}{2}}\\
     &\lesssim |S|^{\frac{1}{2}}\beta_k^{-1}\gamma_1(S)^{-1}\|B_{k,S}\|_{L^2(\IT)}.\nonumber
\end{align}

\noindent Combining \eqref{ee-S-I110} and \eqref{ee-S-I120}, we have
\begin{align*}
    \frak{I}_1
    &\lesssim |S|^{1/2}(\beta_k)^{-1}\gamma_1(S)^{-1}\left(\|a_{k,S}\|_{L^2(\IT)}+\|B_{k,S}\|_{L^2(\IT)}\right).
\end{align*}

\smallskip

Consider $\frak{I}_2$.  Define
\begin{equation*}
    \frak{Q}_{t_1,t_2}(a_{k,S}) := \left| t_1^2 \mathcal{L}_1 e^{-t_1^2 \mathcal{L}_1} \otimes t_2^2 \mathcal{L}_2 e^{-t_2^2 \mathcal{L}_2} a_{k,S} \right|^2 .
\end{equation*}
For $g_1 \in (100\beta_k\hat{I})^c$ and $g_2 \in (100J)^c$,  we can decompose the square function into four parts:
\begin{align*}
    &|\mathcal{S}_{\mathcal{L}_1, \mathcal{L}_2}(a_{k,S})(g_1, g_2)|^2\\
       &=\left(\int_0^{\ell(I)} \!\!\!\int_0^{\ell(J)}\!\! +\int_0^{\ell(I)}\!\!\! \int_{\ell(J)}^\infty+\int_{\ell(I)}^\infty \!\int_0^{\ell(J)} + \int_{\ell(I)}^\infty \!\int_{\ell(J)}^\infty  \right)\int_{\substack{\rho_1(h_1,g_1)<t_1 \\ \rho_2(h_2,g_2)<t_2}} \frak{Q}_{t_1,t_2}(a_{k,S})(\mathbf{h})\, \frac{\ud \mathbf{h} \ud t_1 \ud t_2}{t_1^{Q_1+1} t_2^{Q_2+1}}
       \\
    &=: \frak{I}_{21} + \frak{I}_{22} + \frak{I}_{23} + \frak{I}_{24}.
\end{align*}

%

For $\frak{I}_{24}$, with the help of Theorem \ref{lemma L log L atom}, we obtain
\begin{align*}
    &\bigg|t_1^2\L_1e^{-t_1^2\L_1}\otimes t_2^2\L_2e^{-t_2^2\L_2}(a_{k,S})(\h)\bigg|\\
    &\qquad=t_1^{-2}t_2^{-2} \left|(t_1^2\L_1)^2e^{-t_1^2\L_1}\otimes (t_2^2\L_2)^2e^{-t_2^2\L_2}b_{k,S}(\h)\right|
   \\
    &\qquad\lesssim t_1^{-2}t_2^{-2}\int_{\mathcal{G}_2}\int_{\mathcal{G}_1}\frac{t_1}{(t_1+\rho_1(h_1,z_1))^{Q_1+1}}\cfrac{t_2}{(t_2+\rho_2(h_2,z_2))^{Q_2+1}}|b_{k,S}(z_1,z_2)|\,{\ud z_1\ud z_2}.
\end{align*}
Observe that $g_1\notin 100\beta_k\hat{I}$, $g_2\notin 100\beta_k\hat{J}$ and $t_i>\rho_i(h_i,g_i)$ for $i=1,2$. Consequently, the quasi-triangle inequality immediately yields 
\begin{equation*}
t_i+\rho_i(h_i,z_i)>\rho_i(h_i,g_i)+\rho_i(h_i,z_i) \gtrsim \rho_i(g_i,z_i),\,i=1,2.
\end{equation*}
 Furthermore, we have $z_1\in 3I$ and $z_2\in 3J$ due to the support of $b_{k,S}.$ 
This implies that $\rho_1(g_1,z_1)\approx\rho_1(g_1,g_I)$ and $\rho_2(g_2,z_2)\approx\rho_2(g_2,g_J)$; here $g_I$ and $g_J$ denote the centers of $I$ and $J$, respectively. Thus, we  obtain
\begin{align*}
   &\left|t_1^2\L_1e^{-t_1^2\L_1}\otimes t_2^2\L_2e^{-t_2^2\L_2}(a_{k,S})(\h)\right|\\
   &\qquad\lesssim \frac{1}{t_1t_2}\frac{1}{\rho_1(g_1,g_I)^{Q_1+1}\rho_2(g_2,g_J)^{Q_2+1}}\int_{\IT}|b_{k,S}(z_1,z_2)|\,{\ud z_1 \ud z_2}\\
    &\qquad\lesssim \frac{1}{t_1t_2}\frac{1}{\rho_1(g_1,g_I)^{Q_1+1}\rho_2(g_2,g_J)^{Q_2+1}}|S|^{\frac{1}{2}}\|b_{k,S}\|_{L^2(\IT)}.
\end{align*}
Substituting this estimate into the integral defining $\frak{I}_{24}$, we have
\begin{align*}
    \frak{I}_{24}&\lesssim |S|\int_{\l(I)}^\infty\int_{\l(J)}^\infty \left(\frac{1}{t_1t_2}\right)^2\|b_{k,S}\|_{L^2(\IT)}^2\frac{1}{\rho_1(g_1,g_I)^{2Q_1+2}\rho_2(g_2,g_J)^{2Q_2+2}}\,\frac{{\ud t_1\ud t_2}}{t_1t_2}\\
    &\lesssim |S|\frac{\l(I)^2\l(J)^2}{\rho_1(g_1,g_I)^{2Q_1+2}\rho_2(g_2,g_J)^{2Q_2+2}}\|\l(I)^{-2}\l(J)^{-2}b_{k,S}\|^2_{L^2(\IT)}.
\end{align*}

The calculations for the remaining terms $\frak{I}_{21},\frak{I}_{22},$ and $\frak{I}_{23}$ are similar to those for $\frak{I}_{24}$.
Gathering the preceding estimates, we obtain
\[
\frak{I}_2\lesssim|S|^{\frac{1}{2}}(\beta_k)^{-1}\gamma_1(S)^{-1}\|\l(I)^{-2}\l(J)^{-2}b_{k,S}\|_{L^2(\IT)}.
\]

To sum $\frak{I}_1$ and $\frak{I}_2$ over all dyadic rectangles $S\in \mathscr{M}(\widetilde{\Omega}_k)$, we  apply Lemma \ref{Journetype} and the Cauchy--Schwarz inequality to obtain
\begin{align}\label{ee-6-sum}
    &\sum_{S\in \mathscr{M}(\widetilde{\Omega}_k)} \frak{I}_1 +\frak{I}_2\nonumber\\
    &\lesssim \beta_k^{-1} \sum_{S\in \mathscr{M}(\widetilde{\Omega}_k)} \gamma_1(S)^{-1} |S|^{\frac{1}{2}} \Big( \|a_{k,S}\|_{L^2(\IT)} + \|B_{k,S}\|_{L^2(\IT)}+ \|\l(I)^{-2}\l(J)^{-2}b_{k,S}\|_{L^2(\IT)} \Big) \nonumber\\
    &\lesssim \beta_k^{-1}  \left( \sum_{S\in\mathscr{M}(\widetilde{\Omega}_k)} \Big( \|a_{k,S}\|_{L^2(\IT)}^2+\|B_{k,S}\|_{L^2(\IT)}^2 + \|\l(I)^{-2}\l(J)^{-2}b_{k,S}\|_{L^2(\IT)}^2 \Big) \right)^{\frac{1}{2}} \nonumber\\
    &\hskip 1.2cm \qquad \times \left( \sum_{S\in\mathscr{M}(\widetilde{\Omega}_k)} \gamma_1(S)^{-2} |S| \right)^{\frac{1}{2}} \\
    &\lesssim \beta_k^{-1} |\Omega_k|^{\frac{1}{2}} \bigg(2^{2k}\|2^{-k}f\|_{L\log^+L(\IT)}\bigg)^{\frac{1}{2}} \nonumber\\
    &\lesssim 2^{-\frac{k}{2(Q_1+Q_2)}}\|f\|_{L\log^+L(\IT)}. \nonumber
\end{align}

Similar estimates hold for  the sums of the remaining terms $\frak{I}_3$ and $\frak{I}_4$. Therefore, we have shown that
\begin{equation*}
    \sum_{S\in \mathscr{M}(\widetilde{\Omega}_k)} \int_{(S^\dagger)^c} |\S(a_{k,S})(\g)|\,\ud \g \lesssim 2^{-\frac{k}{2(Q_1+Q_2)}}\|f\|_{L\log^+L(\IT)}.
\end{equation*}
This establishes \eqref{claim s.1}, thereby completing the proof.

\medskip
 
\noindent

\subsection{Proof of \eqref{estimate1.2} for the double Riesz transform}

 Consider the double Riesz transform $\R:=\nabla_{g_1} \L_1^{-1/2}\otimes \nabla_{g_2} \L_2^{-1/2}$
associated to the operators $\L_i,\,i=1,2$. It is known that
\begin{eqnarray}\label{e8.16}
\R f(\g)={1\over 4 {\pi}}\int_0^{\infty}\int_0^{\infty}\big(\nabla_{g_1} e^{-t_1\L_1}\otimes \nabla_{g_2}
e^{-t_2\L_2}\big)f(\g){{\ud t_1\ud t_2}\over \sqrt{t_1 t_2}}.
\end{eqnarray}

\noindent
It was proved in   \cite{Sikora2004}  that for every $i=1,2,$ the Riesz transform
$\nabla_{g_i} \L_i^{-1/2}$ is bounded on $L^p(\IR )$ for $1<p\leq 2$. See also  \cite{CD, DOY, Ouhabaz2005}.
Hence, by using an iteration argument, the double Riesz transform $\R$
 is bounded on $L^p({\IT})$ for $1<p\leq 2$.

 \medskip

We now prove \eqref{estimate1.2} for $\R$. It suffices to show that for all $f\in L\log^+L({\IT})\bigcap L^2({\IT})$,
\begin{align}\label{Riesz goal}
    \bigg|\big\{\g\in \IT:\,|\R(f)(\g)|>1\big\}\bigg|\leq C\|f\|_{L\log^+L(\IT)}.
\end{align}

For $f\in L\log^+L({\IT})\bigcap L^2({\IT})$,  Theorem \ref{lemma L log L atom} yields the atomic decomposition $f=\sum_ka_k$. Using \eqref{S-k<0} and the $L^2(\IT)$ boundedness of $\R$, we obtain
\begin{align}\label{3.18}
    \left|\left\{\g\in\IT:\,\R\left(\sum_{k\leq 0}a_k\right)(\g)>1\right\}\right|\lesssim \|f\|_{L\log^+L(\IT)}.
\end{align}

It remains to estimate $\R(\sum_{k\geq 1}a_k)$.  For the atom $a_k$ with $k>0$, note that $a_k$ is supported in $\Omega_k^\dagger$, with $|\Omega_k^\dagger|\lesssim \|2^{-k}f\|_{L\log^+L(\IT)}$, and  that $a_k=\sum_{S\in \mathscr{M}(\widetilde{\Omega}_k)  }a_{k,S}$.
For each $S=I\times J\subset\widetilde{\Omega}_k$,  we use the notations $\hat{I}, \hat{J}, S^{\dagger}$ from Section \ref{4.1}.


We will prove that there exists  some $\delta>0$ such that
\begin{align}\label{claim s.2}
    \sum_{S\in \mathscr{M}(\widetilde{\Omega}_k)}\int_{(S^\dagger)^c}|\R(a_{k,S})(\g)|\,{\rm d}\g\lesssim 2^{-\delta k}\|f\|_{L\log^+L(\IT)},\quad \forall k\in\mathbb{N}.
\end{align}
\eqref{Riesz goal} follows immediately from \eqref{claim s.2} and \eqref{3.18}. 

We now prove \eqref{claim s.2}. To this end,  decompose the integral over $(S^\dagger)^c$ as follows.
\begin{align*}
&\int_{(S^\dagger)^c} |\R(a_{k,S})(g_1,g_2)|\ud \g\\
&\leq \left(\int_{(100\beta_k\hat{I})^c}\int_{100J}+\int_{(100\beta_k\hat{I})^c}\int_{(100J)^c}+\int_{(100\beta_k\hat{J})^c}\int_{100I}+\int_{(100\beta_k\hat{J})^c}\int_{(100I)^c}\right)|\R(a_{k,S})(\g)|\ud \g\\
& =:\frak{I}_{11}+\frak{I}_{12}+\frak{I}_{13}+\frak{I}_{14}.
\end{align*}
It suffices to estimate $\frak{I}_{11}$ and $\frak{I}_{12}$, as the other two terms are similar.

For $\frak{I}_{11}$, by H\"older's inequality and the $L^2(\IR_{_2})$ boundedness of Riesz operator $\R$, we have
\begin{align*}
    \frak{I}_{11}&\leq \int_{(100\beta_k\hat{I})^c}|100J|^{\frac{1}{2}}\left(\int_{100J}|\nabla_{g_1}\L_1^{-1/2}(a_{k,S})(g_1,g_2)|^2\,{\ud g_2}\right)^{\frac{1}{2}}\,\ud g_1\\
    &\lesssim |J|^{1/2}\sum_{j=1}^{\infty}\int_{S_j}\left(\int_{100J}|\nabla_{g_1}\L_1^{-1/2}(a_{k,S})(g_1,g_2)|^2\,{\ud g_2}\right)^{\frac{1}{2}}\,\ud g_1\\
    &\lesssim |J|^{1/2}\sum_{j=1}^\infty \left(\int_{S_j}\int_{100J}|\nabla_{g_1}\L_1^{-1/2}(a_{k,S})(g_1,g_2)|^2\,{\ud g_2\ud g_1}\right)^{\frac{1}{2}}|S_j|^{1/2}\\
    &\lesssim |J|^{1/2}\sum_{j=1}^\infty \left(\int_{100J}\bigg(\big\|\nabla_{g_1}\L_1^{-1/2}(a_{k,S})(\cdot,g_2)\big\|_{L^2(S_j)}\bigg)^{2}\ud g_2\right)^{\frac{1}{2}}|S_j|^{1/2},
\end{align*}
where $S_j:=2^{j+1}100\beta_k\hat{I}\setminus 2^j(100\beta_k\hat{I})$.
By the definition of the Riesz transform,  we have
\begin{align*}
    \bigg\|\nabla_{g_1}\L_1^{-1/2}(a_{k,S})(\cdot,g_2)\bigg\|_{L^2(S_j)}&\approx\bigg\|\int_0^\infty t_1\nabla_{g_1}e^{-t_1^2\L_1}(a_{k,S})(\cdot,g_2)\,\cfrac{\ud t_1}{t_1}\bigg\|_{L^2(S_j)}\notag\\
    &\lesssim \int_0^{\l(I)}\bigg\| t_1\nabla_{g_1}e^{-t_1^2\L_1}(a_{k,S})(\cdot,g_2)\bigg\|_{L^2(S_j)}\,\cfrac{\ud t_1}{t_1}\notag\\
    &\qquad +\int_{\l(I)}^\infty \bigg\|t_1\nabla_{g_1}e^{-t_1^2\L_1}(a_{k,S})(\cdot,g_2)\bigg\|_{L^2(S_j)}\,\cfrac{\ud t_1}{t_1}\\
    &=:\frak{I}_{111}+\frak{I}_{112}.
\end{align*}

For $\frak{I}_{111}$, by the Davies--Gaffney estimate of $t_1\nabla_{g_1}e^{-t_1^2\L_1}$ and the support condition $\operatorname{supp}(a_{k,S}(\cdot,g_2))\subset 3I$, we deduce that 
\begin{align*}
    \frak{I}_{111}
    &\lesssim \int_0^{\l(I)}e^{-c\frac{(2^j\beta_k\l(\hat{I}))^2}{t_1^2}}\|a_{k,S}(\cdot,g_2)\|_{L^2(\IR_1)}\,\cfrac{\ud t_1}{t_1}\\
    &\lesssim \int_0^{\l(I)}\cfrac{t_1^{Q_1+2}}{(2^j\beta_k\l(\hat{I}))^{Q_1+2}}\,\cfrac{\ud t_1}{t_1}\|a_{k,S}(\cdot,g_2)\|_{L^2(\IR_1)}\\
    &\lesssim {2^j}^{(-Q_1-2)}\beta_k^{-Q_1-2}\gamma_1(S)^{-Q_1-2}\|a_{k,S}(\cdot,g_2)\|_{L^2(\IR_1)}.
\end{align*}
Next, let us handle the term $\frak{I}_{112}$.  For any $M>0$, we define 
$$D_{k,S}(\g):=\l(I)^{-2M}\l(J)^{-2M}\bigg((\l(I)^2\L_1)^0\otimes (\l(J)^2\L_2)^M(b_{k,S})\bigg)(\g).$$
Applying Theorem \ref{lemma L log L atom} with $M=Q_1+2$ and using the Davies--Gaffney estimate, we have
\begin{align*}
    \frak{I}_{112}
    &\lesssim \int_{\l(I)}^\infty \left\|\left(t_1\nabla_{g_1}e^{-\frac{t_1^2\L_1}{2}}\right)\left(t_1^2\L_1\right)^Me^{-\frac{t_1^2\L_1}{2}}\left(\frac{\l(I)}{t_1}\right)^{2M}D_{k,S}(\cdot,g_2) \right\|_{L^2(S_j)}\,\frac{\ud t_1}{t_1}\\
    &\qquad\lesssim \int_{\l(I)}^\infty e^{-c\frac{(2^j\beta_k\l(\hat{I}))^2}{t_1^2}}\left\|D_{k,S}(\cdot,g_2)\right\|_{L^2(\IR_1)}\l(I)^{2M}t_1^{-2M-1}\,\ud t_1\\
    &\qquad\lesssim \int_{\l(I)}^\infty 2^{-2Nj}\beta_k^{-2N}\l(\hat{I})^{-2N}t_1^{2N-2M-1}\l(I)^{2M}\left\|D_{k,S}(\cdot,g_2)\right\|_{L^2(\IR_1)}\,\ud t_1\\
    &\qquad\approx \gamma_1(S)^{-Q_1-2}2^{-(Q_1+2)j}\beta_k^{-Q_1-2}\left\|D_{k,S}(\cdot,g_2)\right\|_{L^2(\IR_1)},
\end{align*}
where in the third inequality we have chosen $N=\frac{Q_1+2}{2}$. Consequently, substituting these estimates back into $\frak{I}_{11}$, we conclude that
\begin{align*}
   \frak{I}_{11}
   &\lesssim |J|^{1/2}\sum_{j=1}^\infty \gamma_1(S)^{-Q_1-2}\beta_k^{-Q_1-2}2^{(-Q_1-2)j}\left(\|a_{k,S}\|_{L^2(\IT)}+\|D_{k,S}\|_{L^2(\IT)}\right)|S_j|^{1/2}\\
   &\lesssim |S|^{1/2}\gamma_1(S)^{-\frac{Q_1}{2}-2}\beta_k^{-\frac{Q_1}{2}-2}\bigg(\|a_{k,S}\|_{L^2(\IT)}+\|D_{k,S}\|_{L^2(\IT)}\bigg).
\end{align*}

\noindent where the last inequality follows from the estimate 
\[
|S_j|^{1/2}\lesssim  2^{\frac{Q_1j}{2}}\beta_k^{\frac{Q_1}{2}}\gamma_1(S)^{\frac{Q_1}{2}}|I|^{1/2}.
\]

\smallskip

Now we turn to the term $\frak{I}_{12}$. For $g_1\notin 100\beta_k\hat{I}$, $g_2\notin 100J$, we apply H\"older's inequality over the annular regions to obtain
\begin{align*}
    \frak{I}_{12}&\leq \int_{(100\beta_k\hat{I})^c}\int_{(100J)^c}|\R(a_{k,S})(\g)|\,\rm{d}\g\\
    &\lesssim \sum_{j=1}^\infty\sum_{l=1}^{\infty}\int_{S_j}\int_{U_l}|\R(a_{k,S})(\g)|\,{\rm d\g}\\
    &\lesssim \sum_{j=1}^\infty\sum_{l=1}^{\infty}\left(\int_{S_j}\int_{U_l}|\R(a_{k,S})(\g)|^2\,{\rm d\g}\right)^{1/2}|S_j|^{1/2}|U_l|^{1/2},
\end{align*}
where $U_l:=2^{l+1}(100J)\setminus 2^l(100J)$.
Applying Minkowski's inequality, we can dominate the $L^{2}(S_j\times U_l)$ norm by splitting the time integration 
\begin{align*}
    &\left(\int_{S_j}\int_{U_l}\left|\nabla_{g_1}\L_1^{-1/2}\otimes\nabla_{g_2}\L_2^{-1/2}(a_{k,S})(\g)\right|^2\,{\rm d\g}\right)^{1/2} \\
    &\qquad\lesssim\left(\int_{0}^{\l(I)}\int_0^{\l(J)}+\int^\infty_{\l(I)}\int_0^{\l(J)}+\int_0^{\l(I)}\int_{\l(J)}^\infty+\int_{\l(I)}^\infty\int_{\l(J)}^\infty\right)  \\
    &\qquad\qquad\qquad\left\|t_1\nabla_{g_1}e^{-t_1^2\L_1}\otimes t_2\nabla_{g_2}e^{-t_2^2\L_2}(a_{k,S})\right\|_{L^2(S_j\times U_l)}\,\cfrac{{\ud t_1\ud t_2}}{t_1t_2} \\
    &\qquad=:\frak{I}_{121}+\frak{I}_{122}+\frak{I}_{123}+\frak{I}_{124}.
\end{align*}

We only  provide detailed  estimates for $\frak{I}_{121}$ and $\frak{I}_{123}$, since the other two terms are similar.

For $\frak{I}_{123}$, by Theorem \ref{lemma L log L atom}, the Davies--Gaffney estimates for $t_i\nabla e^{-t_i^2\L_i}\,(i=1,2)$, and the support condition of $a_{k,S}$, we have
\begin{align*}
    &\frak{I}_{123}\\
    &\lesssim \int_0^{\l(I)}\int_{\l(J)}^\infty e^{-c(\frac{2^{j}\beta_k \l(\hat{I})}{t_1})^2}e^{-c(\frac{2^{l}\l(J)}{t_2})^2}\|\l(I)^{-2M}\l(J)^{-2M}b_{k,S}(\g)\|_{L^2(\IT)}\left(\cfrac{\l(I)\l(J)}{t_1t_2}\right)^{2M}\cfrac{{\ud t_2\ud t_1}}{t_1t_2}\\
        &\lesssim \int_0^{\l(I)}\int_{\l(J)}^\infty t_1^{2m-2M-1}t_2^{2n-2M-1}{\ud t_2\ud t_1}\|\l(I)^{-2M}\l(J)^{-2M}b_{k,S}\|_{L^2(\IT)}\\
    &\hskip3cm \times 2^{-2mj}2^{-2nl}\beta_k^{-2m}l(\hat{I})^{-2m}\l(J)^{-2n}\l(I)^{2M}\l(J)^{2M}\\
    &\lesssim 2^{-2mj}\beta_k^{-2m}2^{-2nl}\gamma_1(S)^{-2m}\|\l(I)^{-2M}\l(J)^{-2M}b_{k,S}\|_{L^2(\IT)},
\end{align*}
where $m>M$ and $n<M$ are to be chosen later.

For $\frak{I}_{121}$, the Davies--Gaffney estimate for $t_i\nabla e^{-t_i^2\L_i},\,i=1,2$ and the support condition of $a_{k,S}$  yield
\begin{align*}
    \frak{I}_{121}&\lesssim \int_0^{\l(I)}\int_0^{\l(J)}e^{-\frac{(2^j\beta_kl(\hat{I}))^2}{t_1^2}}e^{-c\frac{2^{2l}\l(J)^2}{t_2^2}}\|a_{k,S}\|_{L^2(\IT)}\,\cfrac{{\ud t_2\ud t_1}}{t_2t_1}\\
    &\lesssim 2^{-2mj}2^{-2nl}\beta_k^{-2m}\gamma_1(S)^{-2m}\|a_{k,S}\|_{L^2(\IT)},
\end{align*}
for any $m,n>0$.

Since $(|S_j||U_l|)^{1/2}\lesssim 2^{\frac{Q_1j}{2}}2^{\frac{Q_2l}{2}}\gamma_1(S)^{\frac{Q_1}{2}}\beta_k^{\frac{Q_1}{2}}|S|^{1/2}$, we deduce that
\begin{align*}
&\sum_{j=1}^\infty\sum_{l=1}^\infty(\frak{I}_{121}+\frak{I}_{123})|S_j|^{1/2}|U_l|^{1/2}\\
    &\lesssim \sum_{j=1}^\infty\sum_{l=1}^\infty2^{(-2m+\frac{Q_1}{2})j}2^{(-2n+\frac{Q_2}{2})l}\beta_k^{-2m+\frac{Q_1}{2}}\gamma_1(S)^{-2m+\frac{Q_1}{2}}|S|^{1/2}\\
    &\quad\quad\qquad \times\left(\|a_{k,S}\|_{L^2(\IT)}+\|(\l(I)\l(J))^{-2M}b_{k,S}\|_{L^2(\IT)}\right)\\
    &\lesssim|S|^{1/2}\beta_k^{-2m+\frac{Q_1}{2}} \gamma_1(S)^{-2m+\frac{Q_1}{2}}\left(\|a_{k,S}\|_{L^2(\IT)}+\|(\l(I)\l(J))^{-2M}b_{k,S}\|_{L^2(\IT)}\right)
\end{align*}
which  holds as long as $m>{Q_1}/{4}$ and $n>{Q_2}/{4}$. In fact, we can take $m>M>n>\max\{{Q_1}/{4},{Q_2}/{4}\}$, which is always achievable.



By an argument similar to that of \eqref{ee-6-sum}, we deduce \eqref{claim s.2}, and thus complete the proof.

\medskip
\section{Remarks on the Fefferman--Stein type inequality in Euclidean spaces}

\setcounter{equation}{0}

We first record the Euclidean consequence of Theorem \ref{single fefferman-Stein}.
Since \(\mathbb R^n\) is a special case of a stratified Lie group, our main
result applies to Schr\"odinger operators \(L=-\Delta+V\) on \(\mathbb R^n\),
where \(V\geq 0\) and \(V\in L^1_{loc}(\mathbb R^n)\). In this special case,
one can also recover the Fefferman--Stein inequality from the argument of
Merryfield, as used in \cite{merryfield,songJAM2011}. We include the details
below in order to explain why this Euclidean proof does not extend to general
stratified Lie groups.

Let \(\varphi\in C^1_0(\mathbb R^n)\) be nonnegative, radial and nonincreasing.
Assume that
\[
        \operatorname{supp}\varphi\subset B(0,1),
        \qquad
        \int_{\mathbb R^n}\varphi(x)\ud x=1 .
\]
For \(t>0\), write \(\varphi_t(x)=t^{-n}\varphi(x/t)\). Set
\[
        u(x,t)=e^{-t\sqrt L}f(x),
\]
and write
\[
        \widetilde\nabla u=(\partial_tu,\nabla_xu),
        \qquad
        |\widetilde\nabla u|^2=|\partial_tu|^2+|\nabla_xu|^2 .
\]

\begin{lemma}[{\cite[Lemma 2.2]{songJAM2011}}]\label{le7.1}
For all \(f,g\in L^2(\mathbb R^n)\), one has
\begin{align}
 \iint_{\mathbb R^{n+1}_+}
        \big|t\widetilde\nabla u(x,t)\big|^2
        \big|\varphi_t*g(x)\big|^2
        \,\frac{\ud x\ud t}{t}
 &\leq
        \int_{\mathbb R^n}|f(x)|^2|g(x)|^2\,\ud x        \nonumber\\
 &\quad+
        \iint_{\mathbb R^{n+1}_+}
        |u(x,t)|^2|\psi_t*g(x)|^2
        \,\frac{\ud x\ud t}{t},                         \label{e2.5}
\end{align}
where \(\psi\) is a vector-valued function supported in \(B(0,1)\) and
\(\int_{\mathbb R^n}\psi(x)\,\ud x=0\).
\end{lemma}

We now obtain the Euclidean form of the Fefferman--Stein inequality.

\begin{theorem}\label{Rn fSineqality}
There exist constants \(C>0\) and \(\beta>1\) such that, for every
\(f\in C^\infty_0(\mathbb R^n)\) and every \(\lambda>0\),
\begin{equation}\label{thm7-FS}
\begin{aligned}
\left|\{x\in\mathbb R^n:\,S_{P,L}(f)(x)>\lambda\}\right|
&\leq
C\left|\{x\in\mathbb R^n:\,N^\beta_L(f)(x)>\lambda\}\right|        \\
&\quad+
\frac{C}{\lambda^2}
\int_{\{x:\,N^\beta_L(f)(x)\leq \lambda\}}
        |N^\beta_L(f)(x)|^2\ud x .
\end{aligned}
\end{equation}
\end{theorem}

\begin{proof}
It is enough to prove the estimate with \(S_{P,L}\) replaced by the full
gradient area function
\[
        S_{\widetilde\nabla,L}(f)(x)
        =
        \left(
        \iint_{\Gamma(x)}
        |t\widetilde\nabla e^{-t\sqrt L}f(y)|^2
        \,\frac{\ud y\ud t}{t^{n+1}}
        \right)^{1/2},
\]
because
\[
        t\sqrt L e^{-t\sqrt L}f=-t\partial_t e^{-t\sqrt L}f .
\]

Fix \(\lambda>0\). Let
\[
        E_\lambda
        =
        \{x\in\mathbb R^n:\,N^\beta_L(f)(x)\leq \lambda\},
        \qquad
        E_\lambda^c
        =
        \{x\in\mathbb R^n:\,N^\beta_L(f)(x)>\lambda\}.
\]
Let
\[
        A_\lambda
        =
        \left\{
        x\in\mathbb R^n:\,
        \mathcal M(\chi_{E_\lambda^c})(x)
        \leq {1\over 10C_0}
        \right\},
\]
where \(C_0\) is chosen large enough so that the standard maximal estimate
used in \eqref{poisson controlled by H--L} also controls convolution with
\(\varphi_t\). Finally set
\[
        W=\bigcup_{y\in A_\lambda}\Gamma(y).
\]
By Chebyshev's inequality and the \(L^2\)-boundedness of the Hardy--Littlewood
maximal operator,
\[
\begin{aligned}
\left|\{x:\,S_{\widetilde\nabla,L}(f)(x)>\lambda\}\right|
&\leq
|A_\lambda^c|
+
{1\over \lambda^2}
\int_{A_\lambda}|S_{\widetilde\nabla,L}(f)(x)|^2\ud x                  
\lesssim
|E_\lambda^c|
+
{1\over \lambda^2}
\int_{A_\lambda}|S_{\widetilde\nabla,L}(f)(x)|^2\ud x .
\end{aligned}
\]
Thus it remains to prove
\begin{equation}\label{e-7-key inequality}
\int_W|\widetilde\nabla u(x,t)|^2t\ud x\ud t
\lesssim
\int_{E_\lambda}|N^\beta_L(f)(x)|^2\ud x
+
\lambda^2|E_\lambda^c|.
\end{equation}
Indeed, by Tonelli's theorem,
\[
        \int_{A_\lambda}|S_{\widetilde\nabla,L}(f)(x)|^2\ud x
        \lesssim
        \int_W|\widetilde\nabla u(x,t)|^2t\ud x\ud t .
\]

We now prove \eqref{e-7-key inequality}. If \((x,t)\in W\), then there exists
\(z\in A_\lambda\) such that \(|x-z|<t\). Since \(\operatorname{supp}\varphi_t
\subset B(x,t)\), we have
\[
        |\varphi_t*\chi_{E_\lambda^c}(x)|
        \leq C_0\mathcal M(\chi_{E_\lambda^c})(z)
        \leq {1\over 10}.
\]
Using \(\int\varphi(x)\,\ud x=1\), this gives
\[
        \varphi_t*\chi_{E_\lambda}(x)
        =
        1-\varphi_t*\chi_{E_\lambda^c}(x)
        \geq {9\over 10}.
\]
Therefore
\[
\int_W|\widetilde\nabla u(x,t)|^2t\ud x\ud t
\lesssim
\iint_{\mathbb R^{n+1}_+}
        |t\widetilde\nabla u(x,t)|^2
        |\varphi_t*\chi_{E_\lambda}(x)|^2
        \,\frac{\ud x\ud t}{t}.
\]
Strictly speaking, if \(E_\lambda\) has infinite measure, one first applies
Lemma \ref{le7.1} to \(\chi_{E_\lambda}\eta_R\), where \(\eta_R\) is a smooth
cutoff which tends to \(1\), and then lets \(R\to\infty\). We omit this standard
approximation below.

By Lemma \ref{le7.1}, with \(g=\chi_{E_\lambda}\), we obtain
\begin{align*}
\int_W|\widetilde\nabla u(x,t)|^2t\ud x\ud t
&\lesssim
\int_{\mathbb R^n}|f(x)|^2\chi_{E_\lambda}(x)\ud x        +
\iint_{\mathbb R^{n+1}_+}
        |u(x,t)|^2|\psi_t*\chi_{E_\lambda}(x)|^2
        \,\frac{\ud x\ud t}{t}.
\end{align*}
Since \(e^{-t\sqrt L}f\to f\) at the boundary and since the nontangential
maximal function dominates the boundary trace, we have
\[
        |f(x)|\leq N^\beta_L(f)(x)
        \quad\text{for a.e. }x\in\mathbb R^n .
\]
Hence
\[
        \int_{\mathbb R^n}|f(x)|^2\chi_{E_\lambda}(x)\ud x
        \leq
        \int_{E_\lambda}|N^\beta_L(f)(x)|^2\ud x .
\]

It remains to estimate the second term. If
\(\psi_t*\chi_{E_\lambda}(x)\neq 0\), then, since \(\operatorname{supp}\psi
\subset B(0,1)\), there exists \(z_0\in B(x,t)\cap E_\lambda\). Taking
\(\beta>1\), we have \((x,t)\in\Gamma^\beta(z_0)\). Therefore
\[
        |u(x,t)|
        =
        |e^{-t\sqrt L}f(x)|
        \leq
        N^\beta_L(f)(z_0)
        \leq \lambda .
\]
It follows that
\begin{align*}
&\iint_{\mathbb R^{n+1}_+}
        |u(x,t)|^2|\psi_t*\chi_{E_\lambda}(x)|^2
        \,\frac{\ud x\ud t}{t}                                      \lesssim
        \lambda^2
        \iint_{\mathbb R^{n+1}_+}
        |\psi_t*\chi_{E_\lambda}(x)|^2
        \,\frac{\ud x\ud t}{t}.
\end{align*}
Since \(\int\psi(x)\ud x=0\), we have
\[
        \psi_t*\chi_{E_\lambda}
        =
        -\psi_t*\chi_{E_\lambda^c}.
\]
Thus the classical \(L^2\) Littlewood--Paley estimate gives
\[
\begin{aligned}
\iint_{\mathbb R^{n+1}_+}
        |\psi_t*\chi_{E_\lambda}(x)|^2
        \,\frac{\ud x\ud t}{t}
&=
\iint_{\mathbb R^{n+1}_+}
        |\psi_t*\chi_{E_\lambda^c}(x)|^2
        \,\frac{\ud x\ud t}{t}       
\lesssim
        \|\chi_{E_\lambda^c}\|_{L^2(\mathbb R^n)}^2
        =
        |E_\lambda^c|.
\end{aligned}
\]
Combining the last estimates proves \eqref{e-7-key inequality}, and hence
\eqref{thm7-FS}.
\end{proof}

\begin{remark}
The proof above uses a Euclidean feature which has no direct analogue on a
general stratified Lie group. The key point in Lemma \ref{le7.1} is the
construction of a compactly supported bump function \(\varphi\) for which
\[
        \partial_t(\varphi_t*f)
        =
        -\sum_{j=1}^n
        \partial_{x_j}\big((\rho_j)_t*f\big),
        \qquad
        \rho_j(x)=x_j\varphi(x).
\]
This identity is based on ordinary Euclidean convolution, Euclidean dilations
and the usual partial derivatives. It allows one to integrate by parts and
replace derivatives of the averaging function by a square function with
mean-zero kernels.

For a general stratified Lie group, the horizontal vector fields are
noncommutative and the natural dilations are anisotropic. There is no direct
replacement of the above compactly supported Euclidean identity with the same
support and cancellation properties. Moreover, when \(V\not\equiv 0\), the
Poisson semigroup associated with \(L=-\Delta+V\) is not a convolution
semigroup. Therefore the Merryfield argument does not give the desired
Fefferman--Stein inequality in the stratified Lie group setting. This is why
the proof of Theorem \ref{single fefferman-Stein} uses a different cutoff, built from
the sub-Laplacian Poisson semigroup \(e^{-t\sqrt{-\Delta}}\chi_{E_\lambda}\),
together with the truncation and weak derivative argument developed in
Section \ref{section 3.2}.
\end{remark}

We also point out that Theorems \ref{prod F-S thm}, \ref{lemma L log L atom} and \ref{FS-Product space} are new even in the Euclidean spaces.

\bigskip
\bigskip

{\bf Acknowledgments.}   The research is supported by National Key R$\&$D Program of China 2022YFA1005700.  J. Li is  supported by ARC DP 260100485. L. Song  is supported by NNSF of China (No. 12471097).  L. Yan is supported by NNSF of China (No. 12571111). 

\vskip 1cm

\bibliographystyle{plain}
		\bibliography{ref.bib} 

@article {AT,
    AUTHOR = {Auscher, P. and Tchamitchian, P. },
     TITLE = {Square root problem for divergence operators and related
              topics},
   JOURNAL = {Ast\'erisque},
  FJOURNAL = {Ast\'erisque},
    NUMBER = {249},
      YEAR = {1998},
     PAGES = {viii+172},
      ISSN = {0303-1179,2492-5926},
   MRCLASS = {47F05 (35A25 35J99 47G30)},
  MRNUMBER = {1651262},
MRREVIEWER = {Xuan\ Thinh\ Duong},
}

@article {CF1980,
    AUTHOR = {Chang, S-Y.A. and Fefferman, R.},
     TITLE = {A continuous version of duality of {$H\sp{1}$}\ with {BMO} on
              the bidisc},
   JOURNAL = {Ann. of Math. (2)},
  FJOURNAL = {Annals of Mathematics. Second Series},
    VOLUME = {112},
      YEAR = {1980},
    NUMBER = {1},
     PAGES = {179--201},
      ISSN = {0003-486X},
   MRCLASS = {32A35 (42B30)},
  MRNUMBER = {584078},
MRREVIEWER = {A.\ B.\ Aleksandrov},
}

@article {CF1982,
    AUTHOR = {Chang, S-Y.A. and Fefferman, R.},
     TITLE = {The {C}alder\'on-{Z}ygmund decomposition on product domains},
   JOURNAL = {Amer. J. Math.},
  FJOURNAL = {American Journal of Mathematics},
    VOLUME = {104},
      YEAR = {1982},
    NUMBER = {3},
     PAGES = {455--468},
      ISSN = {0002-9327,1080-6377},
   MRCLASS = {42B30 (32A35)},
  MRNUMBER = {658542},
MRREVIEWER = {Douglas\ Kurtz},
       }

@article {CD,
    AUTHOR = {Coulhon, T. and Duong, X.T.},
     TITLE = {Riesz transforms for {$p>2$}},
   JOURNAL = {C. R. Acad. Sci. Paris S\'er. I Math.},
  FJOURNAL = {Comptes Rendus de l'Acad\'emie des Sciences. S\'erie I.
              Math\'ematique},
    VOLUME = {332},
      YEAR = {2001},
    NUMBER = {11},
     PAGES = {975--980},
      ISSN = {0764-4442},
   MRCLASS = {58J35 (42B25 47D06 47F05)},
  MRNUMBER = {1838122},
MRREVIEWER = {Emmanuel\ Russ},
       }

@book {FoSt,
    AUTHOR = {Folland, G.B. and Stein, E.M.},
     TITLE = {Hardy spaces on homogeneous groups},
    SERIES = {Mathematical Notes},
    VOLUME = {28},
 PUBLISHER = {Princeton University Press, Princeton, NJ; University of Tokyo
              Press, Tokyo},
      YEAR = {1982},
     PAGES = {xii+285},
      ISBN = {0-691-08310-X},
   MRCLASS = {43A85 (22E45 42B30)},
  MRNUMBER = {657581},
MRREVIEWER = {Daryl\ Geller},
}

@article {SY2010,
    AUTHOR = {Song, L. and Yan, L.X.},
     TITLE = {Riesz transforms associated to {S}chr\"odinger operators on
              weighted {H}ardy spaces},
   JOURNAL = {J. Funct. Anal.},
  FJOURNAL = {Journal of Functional Analysis},
    VOLUME = {259},
      YEAR = {2010},
    NUMBER = {6},
     PAGES = {1466--1490},
      ISSN = {0022-1236,1096-0783},
   MRCLASS = {35J10 (42B20 47D08 47F05)},
  MRNUMBER = {2659768},
MRREVIEWER = {Shijun\ Zheng},
}

@article {CS,
    AUTHOR = {Coulhon, T. and Sikora, A.},
     TITLE = {Gaussian heat kernel upper bounds via the
              {P}hragm\'en-{L}indel\"of theorem},
   JOURNAL = {Proc. Lond. Math. Soc. (3)},
  FJOURNAL = {Proceedings of the London Mathematical Society. Third Series},
    VOLUME = {96},
      YEAR = {2008},
    NUMBER = {2},
     PAGES = {507--544},
      ISSN = {0024-6115,1460-244X},
   MRCLASS = {35K05 (58J35)},
  MRNUMBER = {2396848},
      }

@article {DOY,
    AUTHOR = {Duong, X.T. and Ouhabaz, E.M. and Yan, L.X.},
     TITLE = {Endpoint estimates for {R}iesz transforms of magnetic
              {S}chr\"odinger operators},
   JOURNAL = {Ark. Mat.},
  FJOURNAL = {Arkiv f\"or Matematik},
    VOLUME = {44},
      YEAR = {2006},
    NUMBER = {2},
     PAGES = {261--275},
      ISSN = {0004-2080,1871-2487},
   MRCLASS = {35Q40 (31B10 42B20 47F05 81Q10)},
  MRNUMBER = {2292721},
MRREVIEWER = {Andrei\ B.\ Bogatyr\"ev},
       DOI = {10.1007/s11512-006-0021-x},
       URL = {https://doi.org/10.1007/s11512-006-0021-x},
}

@book {Brezis,
    AUTHOR = {Brezis, H.},
     TITLE = {Functional analysis, {S}obolev spaces and partial differential
              equations},
    SERIES = {Universitext},
 PUBLISHER = {Springer, New York},
      YEAR = {2011},
     PAGES = {xiv+599},
      ISBN = {978-0-387-70913-0},
   MRCLASS = {35-01 (46-01 46E35 46N20 47F05)},
  MRNUMBER = {2759829},
MRREVIEWER = {Vicen\c tiu\ D.\ R\u adulescu},
}

@article {Feff1986,
    AUTHOR = {Fefferman, R.},
     TITLE = {Calder\'on-{Z}ygmund theory for product domains: {$H^p$}
              spaces},
   JOURNAL = {Proc. Nat. Acad. Sci. U.S.A.},
  FJOURNAL = {Proceedings of the National Academy of Sciences of the United
              States of America},
    VOLUME = {83},
      YEAR = {1986},
    NUMBER = {4},
     PAGES = {840--843},
      ISSN = {0027-8424},
   MRCLASS = {42B30 (42B20)},
  MRNUMBER = {828217},
MRREVIEWER = {M.\ Cotlar},
      }

@article {F1987,
    AUTHOR = {Fefferman, R.},
     TITLE = {Harmonic analysis on product spaces},
   JOURNAL = {Ann. of Math. (2)},
  FJOURNAL = {Annals of Mathematics. Second Series},
    VOLUME = {126},
      YEAR = {1987},
    NUMBER = {1},
     PAGES = {109--130},
      ISSN = {0003-486X,1939-8980},
   MRCLASS = {42B20 (42B30 47B38)},
  MRNUMBER = {898053},
MRREVIEWER = {Norman\ J.\ Weiss},
      }

@article {HLCL2010,
    AUTHOR = {Han, Y.S. and Lee, M.Y. and Lin, C.C. and Lin,
              Y.C.},
     TITLE = {Calder\'on-{Z}ygmund operators on product {H}ardy spaces},
   JOURNAL = {J. Funct. Anal.},
  FJOURNAL = {Journal of Functional Analysis},
    VOLUME = {258},
      YEAR = {2010},
    NUMBER = {8},
     PAGES = {2834--2861},
      ISSN = {0022-1236,1096-0783},
   MRCLASS = {42B20 (42B25 47G10)},
  MRNUMBER = {2593346},
MRREVIEWER = {Leszek\ Skrzypczak},
       
}

@article {FS1972,
    AUTHOR = {Fefferman, C. and Stein, E.M.},
     TITLE = {{$H\sp{p}$} spaces of several variables},
   JOURNAL = {Acta Math.},
  FJOURNAL = {Acta Mathematica},
    VOLUME = {129},
      YEAR = {1972},
    NUMBER = {3-4},
     PAGES = {137--193},
      ISSN = {0001-5962,1871-2509},
   MRCLASS = {42A40 (30A78 42A18 42A92)},
  MRNUMBER = {447953},
MRREVIEWER = {Alberto\ Torchinsky},
       DOI = {10.1007/BF02392215},
       URL = {https://doi.org/10.1007/BF02392215},
}

@article {FSt1982,
    AUTHOR = {Fefferman, R. and Stein, E.M.},
     TITLE = {Singular integrals on product spaces},
   JOURNAL = {Adv. in Math.},
  FJOURNAL = {Advances in Mathematics},
    VOLUME = {45},
      YEAR = {1982},
    NUMBER = {2},
     PAGES = {117--143},
      ISSN = {0001-8708},
   MRCLASS = {42B20 (44A35)},
  MRNUMBER = {664621},
MRREVIEWER = {Akihiko\ Miyachi},
       DOI = {10.1016/S0001-8708(82)80001-7},
       URL = {https://doi.org/10.1016/S0001-8708(82)80001-7},
}

@article {H,
    AUTHOR = {H\"ormander, L.},
     TITLE = {Estimates for translation invariant operators in {$L\sp{p}$}\
              spaces},
   JOURNAL = {Acta Math.},
  FJOURNAL = {Acta Mathematica},
    VOLUME = {104},
      YEAR = {1960},
     PAGES = {93--140},
      ISSN = {0001-5962,1871-2509},
   MRCLASS = {46.00 (42.00)},
  MRNUMBER = {121655},
MRREVIEWER = {I.\ I.\ Hirschman, Jr.},
     }

@article{HLMMY,
  author = {Hofmann, S. and Lu, G.Z. and Mitrea, D. and Mitrea, M. and Yan, L.X.},
     TITLE = {Hardy spaces associated to non-negative self-adjoint operators
              satisfying {D}avies-{G}affney estimates},
   JOURNAL = {Mem. Amer. Math. Soc.},
  FJOURNAL = {Memoirs of the American Mathematical Society},
    VOLUME = {214},
      YEAR = {2011},
    NUMBER = {1007},
     PAGES = {vi+78},
}

@article {HLL2016,
    AUTHOR = {Han, Y.S. and Li, J. and Lin, C.-C.},
     TITLE = {Criterion of the {$L^2$} boundedness and sharp endpoint
              estimates for singular integral operators on product spaces of
              homogeneous type},
   JOURNAL = {Ann. Sc. Norm. Super. Pisa Cl. Sci. (5)},
  FJOURNAL = {Annali della Scuola Normale Superiore di Pisa. Classe di
              Scienze. Serie V},
    VOLUME = {16},
      YEAR = {2016},
    NUMBER = {3},
     PAGES = {845--907},
      ISSN = {0391-173X,2036-2145},
   MRCLASS = {42B20 (42B25)},
  MRNUMBER = {3618079},
MRREVIEWER = {Michael\ T.\ Lacey},
}

@article {J,
    AUTHOR = {Journ\'e, J.-L.},
     TITLE = {A covering lemma for product spaces},
   JOURNAL = {Proc. Amer. Math. Soc.},
  FJOURNAL = {Proceedings of the American Mathematical Society},
    VOLUME = {96},
      YEAR = {1986},
    NUMBER = {4},
     PAGES = {593--598},
      ISSN = {0002-9939,1088-6826},
   MRCLASS = {42B20},
  MRNUMBER = {826486},
MRREVIEWER = {Sun\ Yung A. Chang},
      }

@book {Ouhabaz2005,
    AUTHOR = {Ouhabaz, E.M.},
     TITLE = {Analysis of heat equations on domains},
    SERIES = {London Mathematical Society Monographs Series},
    VOLUME = {31},
 PUBLISHER = {Princeton University Press, Princeton, NJ},
      YEAR = {2005},
     PAGES = {xiv+284},
      ISBN = {0-691-12016-1},
   MRCLASS = {35-02 (35J70 35K20 47D06 47D07 47F05)},
  MRNUMBER = {2124040},
MRREVIEWER = {Sergey\ G.\ Pyatkov},
}

@article {Sikora2004,
    AUTHOR = {Sikora, A.},
     TITLE = {Riesz transform, {G}aussian bounds and the method of wave
              equation},
   JOURNAL = {Math. Z.},
  FJOURNAL = {Mathematische Zeitschrift},
    VOLUME = {247},
      YEAR = {2004},
    NUMBER = {3},
     PAGES = {643--662},
      ISSN = {0025-5874,1432-1823},
   MRCLASS = {58J35 (35L05 42B20)},
  MRNUMBER = {2114433},
MRREVIEWER = {Thierry\ Coulhon},
      }

@book {yo,
    AUTHOR = {Yosida, K.},
     TITLE = {Functional analysis},
    SERIES = {Classics in Mathematics},
      NOTE = {Reprint of the sixth (1980) edition},
 PUBLISHER = {Springer-Verlag, Berlin},
      YEAR = {1995},
     PAGES = {xii+501},
      ISBN = {3-540-58654-7},
   MRCLASS = {46-01 (47-01)},
  MRNUMBER = {1336382},
       DOI = {10.1007/978-3-642-61859-8},
       URL = {https://doi.org/10.1007/978-3-642-61859-8},
}

@article{merryfield,
    AUTHOR = {Merryfield, K.G.},
     TITLE = {On the area integral, {C}arleson measures and {$H^p$} in the
              polydisc},
   JOURNAL = {Indiana Univ. Math. J.},
  FJOURNAL = {Indiana University Mathematics Journal},
    VOLUME = {34},
      YEAR = {1985},
    NUMBER = {3},
     PAGES = {663--685},
      ISSN = {0022-2518,1943-5258},
   MRCLASS = {42B30 (42B25)},
  MRNUMBER = {794581},
MRREVIEWER = {Mario\ Milman},
}

@article{liji2023,
      title={On {F}efferman--{S}tein type inequality on Shilov boundaries and applications}, 
      author={Li, J.},
      year={2024},
JOURNAL = {Ann. Sc. Norm. Super. Pisa Cl. Sci. (5)}
}

@article{fefferman1986,
    AUTHOR = {Fefferman, R.},
     TITLE = {A note on a lemma of {Z}\'o},
   JOURNAL = {Proc. Amer. Math. Soc.},
  FJOURNAL = {Proceedings of the American Mathematical Society},
    VOLUME = {96},
      YEAR = {1986},
    NUMBER = {2},
     PAGES = {241--246},
}

@article{CLLP2025,
  title={An endpoint estimate for product singular integral operators on stratified Lie groups},
  author={Cowling, M.G. and Lee, M.Y. and Li, J. and Pipher, J.},
  journal={Canad. J. Math.},
  PAGES={1--32},
  year={2025},
   VOLUME = {482},
  publisher={Canadian Mathematical Society},
}

@article {songJAM2011,
    AUTHOR = {Song, L. and Tan, C.Q. and Yan, L.X.},
     TITLE = {An atomic decomposition for {H}ardy spaces associated to
              {S}chr\"odinger operators},
   JOURNAL = {J. Aust. Math. Soc.},
  FJOURNAL = {Journal of the Australian Mathematical Society},
    VOLUME = {91},
      YEAR = {2011},
    NUMBER = {1},
     PAGES = {125--144},
      ISSN = {1446-7887,1446-8107},
   MRCLASS = {42B30 (35J10 42B25 47F05)},
  MRNUMBER = {2844951},
MRREVIEWER = {Tuomas\ P.\ Hyt\"onen},
     }

@article {HKMP2015,
    AUTHOR = {Hofmann, S. and Kenig, C. and Mayboroda, S. and
              Pipher, J.},
     TITLE = {Square function/non-tangential maximal function estimates and
              the {D}irichlet problem for non-symmetric elliptic operators},
   JOURNAL = {J. Amer. Math. Soc.},
  FJOURNAL = {Journal of the American Mathematical Society},
    VOLUME = {28},
      YEAR = {2015},
    NUMBER = {2},
     PAGES = {483--529},
      ISSN = {0894-0347,1088-6834},
   MRCLASS = {35J25 (35B45 42B20 42B25 42B37)},
  MRNUMBER = {3300700},
MRREVIEWER = {Lubomira\ G.\ Softova},
  }

@incollection{CowlingFanLiyan2025,
    AUTHOR = {Cowling, M.G. and Fan, Z.J. and Li, J. and Yan, L.X.},
     TITLE = {Characterizations of product {H}ardy spaces on stratified
              groups by singular integrals and maximal functions},
 BOOKTITLE = {The mathematical heritage of {G}uido {W}eiss},
    SERIES = {Appl. Numer. Harmon. Anal.},
     pages = {193--227},
 PUBLISHER = {Birkh\"auser/Springer, Cham},
      YEAR = {2025},
      ISBN = {978-3-031-76792-0; 978-3-031-76793-7},
   MRCLASS = {42B20 (42B25 42B30 42B35 43A80)},
  MRNUMBER = {4934898},
    }

@article {JS1986,
    AUTHOR = {Jerison, D.S. and S\'anchez-Calle, A.},
     TITLE = {Estimates for the heat kernel for a sum of squares of vector
              fields},
   JOURNAL = {Indiana Univ. Math. J.},
  FJOURNAL = {Indiana University Mathematics Journal},
    VOLUME = {35},
      YEAR = {1986},
    NUMBER = {4},
     PAGES = {835--854},
      ISSN = {0022-2518,1943-5258},
   MRCLASS = {58G11 (32F25 35H05 35K05)},
  MRNUMBER = {865430},
MRREVIEWER = {David\ S.\ Tartakoff},
       DOI = {10.1512/iumj.1986.35.35043},
       URL = {https://doi.org/10.1512/iumj.1986.35.35043},
}

@book {Goldstein1985,
    AUTHOR = {Goldstein, J.A.},
     TITLE = {Semigroups of linear operators and applications},
    SERIES = {Oxford Mathematical Monographs},
 PUBLISHER = {The Clarendon Press, Oxford University Press, New York},
      YEAR = {1985},
     PAGES = {x+245},
      ISBN = {0-19-503540-2},
   MRCLASS = {47D05 (34K30 35R20)},
  MRNUMBER = {790497},
MRREVIEWER = {H.\ O.\ Fattorini},
}

@article{Estein1958,
 ISSN = {00029947, 10886850},
 URL = {http://www.jstor.org/stable/1993226},
 author = {Stein, E.M.},
 journal = {Trans. Amer. Math. Soc.
},
 number = {2},
 pages = {430--466},
 publisher = {American Mathematical Society},
 title = {On the Functions of {L}ittlewood-{P}aley, {L}usin, and {M}arcinkiewicz},
 urldate = {2025-12-07},
 volume = {88},
 year = {1958}
}

@article {hk2012,
    AUTHOR = {Hyt\"onen, T. and Kairema, A.},
     TITLE = {Systems of dyadic cubes in a doubling metric space},
   JOURNAL = {Colloq. Math.},
  FJOURNAL = {Colloquium Mathematicum},
    VOLUME = {126},
      YEAR = {2012},
    NUMBER = {1},
     PAGES = {1--33},
      ISSN = {0010-1354,1730-6302},
   MRCLASS = {42B25 (60D05)},
  MRNUMBER = {2901199},
MRREVIEWER = {Raymond\ H.\ Cox},
       DOI = {10.4064/cm126-1-1},
       URL = {https://doi.org/10.4064/cm126-1-1},
}

\end{document}